\documentclass[10pt]{amsart}
\usepackage{amsmath, amssymb, amsthm, slashed, mathtools}
\usepackage{fullpage}

\usepackage{amsthm, amsopn, amsfonts, xypic}
\usepackage[colorlinks=true]{hyperref}
\usepackage{hyperref}

\newcommand{\RR}{\mathbb{R}}

\newcommand{\CC}{\mathbb{C}}
\newcommand{\les}{\leqslant}
\newcommand{\lesa}{\lesssim}
\newcommand{\g}{\geqslant}
\newcommand{\mc}[1]{\mathcal{#1}}
\newcommand{\p}{\partial}
\newcommand{\lr}[1]{ \langle #1 \rangle}
\newcommand{\eref}[1]{(\ref{#1})}
\newcommand{\ind}{\mathbbold{1}}
\DeclareSymbolFont{bbold}{U}{bbold}{m}{n}
\DeclareSymbolFontAlphabet{\mathbbold}{bbold}

\newcommand{\sF}{ {{}^\star \! F} }
\newcommand{\sD}{ {\, {\slash\!\!\!\! D  }}}
\newcommand{\balpha}{ \underline{\alpha}}
\newcommand{\bL}{\underline{L}}
\newcommand{\lp}[2]{\Vert \, #1 \, \Vert_{#2}}
\newcommand{\ull}{{\underline{L}}}

\newcommand{\tp}{\langle t+r \rangle}
\newcommand{\tm}{\langle t-r \rangle}

\newcommand{\beq}{\begin{equation}}\newcommand{\eq}{\end{equation}}
\newcommand{\beqs}{\begin{equation*}}\newcommand{\eqs}{\end{equation*}}
\def\pa{\partial}

\newcommand{\Lie}[1]{\mathcal{L}_{#1}}

\newtheorem{Theorem}{Theorem}[section]
\newtheorem{Lemma}[Theorem]{Lemma}\newtheorem{lemma}[Theorem]{Lemma}
\newtheorem{Proposition}[Theorem]{Proposition}
\newtheorem*{Remark}{Remark}

\numberwithin{equation}{section}
\begin{document}

\title{Asymptotic Behavior of the Maxwell-Klein-Gordon system}%
\author{Timothy Candy}%
\email{tcandy@math.uni-bielefeld.de}%
\address[T.~Candy]{Universit\"at Bielefeld, Fakult\"at f\"ur Mathematik,
  Postfach 100131, 33501 Bielefeld, Germany}%

\author{Christopher Kauffman}%
\email{kauffman@math.jhu.edu}%
\address[C.~Kauffman]{Johns Hopkins University,  Krieger Hall, 3400 N.~Charles Street, Baltimore, MD 21218, US}%

\author{Hans Lindblad}%
\email{lindblad@math.jhu.edu}%
\address[H.~Lindblad]{Johns Hopkins University,  Krieger Hall, 3400 N.~Charles Street, Baltimore, MD 21218, US}%

\thanks{T.C. acknowledges financial support by the DFG through the CRC ``Taming uncertainty and
profiting from randomness and low regularity in analysis, stochastics
and their applications''. C.K and H.L. were supported in part by NSF Grant DMS-1500925}


\maketitle
\begin{abstract}
 In previous work on the Maxwell-Klein-Gordon system first global existence and then decay estimates have been shown. Here we show that the Maxwell-Klein-Gordon system in the Lorenz gauge satisfies the {\it weak null condition} and give detailed asymptotics for the scalar field and the potential. These asymptotics have two parts, one wave like along outgoing light cones at null infinity, and one homogeneous inside the light cone at time like infinity. Here the charge plays a crucial role in imposing an oscillating factor in the asymptotic system for the field, and in the null asymptotics for the potential. The Maxwell-Klein-Gordon system, apart from being of interest in its own right, also provides a simpler semi-linear model of the quasi-linear Einstein's equations where similar asymptotic results have previously been obtained in wave coordinates.
 \end{abstract}

\section{Introduction}\label{sec:Intro}

The Maxwell-Klein-Gordon equation for a scalar field $\phi: \RR^{1+3} \rightarrow \CC$ and a potential $A_\alpha: \RR^{1+3} \rightarrow \RR$ is given by
    \begin{equation}\label{eq:waveequation}
        \begin{split}
        D^\alpha D_\alpha \phi &=0 \\
            \p^\beta F_{\alpha \beta} &= J_\alpha
        \end{split}
    \end{equation}
where the covariant derivative is given by $D_\alpha = \p_\alpha + i A_\alpha$, the curvature is defined as $F_{\alpha \beta} = \p_\alpha A_\beta - \p_\beta A_\alpha$, and $J_\alpha = \Im( \phi \overline{D_\alpha \phi})$ is the current. We take $x^0 = t$, and indices are raised and lowered with respect to the Minkowski metric $m = \text{diag }(-1, 1, 1, 1)$. The Einstein summation convention is in effect with Greek indices summed over $\alpha = 0, \dots, 3$, and Latin indices summed over the spatial variables $j=1, 2, 3$. Thus $\p^\alpha = m^{\alpha \beta } \p_{\beta}$ and $\p^0 = - \p_t$.

The energy-momentum tensor of this system is given by
\begin{equation*}
Q_{\alpha\beta} = D_\alpha\phi D_\beta\phi - \frac12 m_{\alpha\beta}D_\gamma \phi D^\gamma \phi + \frac12 F_{\alpha\gamma}F_\beta^\gamma + \frac12 \sF_{\alpha\gamma}\sF_\beta^\gamma
\end{equation*}
where $\sF_{\alpha \beta} = \frac{1}{2} \epsilon_{\alpha \beta}^{\;\;\;\;\mu \nu} F_{\mu \nu}$ is the Hodge star operator applied to $F$, and $\epsilon_{\alpha \beta \mu \nu}$ is the volume form on Minkowski space-time $\RR^{1+3}$. A computation shows that the energy $Q_{\alpha \beta}$ is divergence-free for solutions $(\phi, A)$ to \eqref{eq:waveequation}, as the current $J$ satisfies $\p^\alpha J_\alpha = 0$.
Moreover, integrating the divergence of the current over $\RR^3$, we conclude that the charge
    $$ \mathbf{q}   = \int_{\RR^3} J_0 dx,$$
 is conserved. The charge plays a key role in the large time behaviour of the solution $(\phi, A_\mu)$, as it causes a long range correction to the asymptotics of both the scalar field $\phi$, and the gauge $A$.
\\

The system \eref{eq:waveequation} does not uniquely determine $\phi$ and $A$. In particular, for any function $\psi$, the gauge transform $\tilde{A} = A+d\psi$ gives the same field $F$, and moreover, letting $\tilde{\phi}= e^{i \psi} \phi$, if $(\phi, A)$ solve \eref{eq:waveequation}, then $(\tilde{\phi}, \tilde{A})$ also gives a solution to \eref{eq:waveequation}.  In this article, we fix the gauge by imposing the Lorenz gauge condition
\begin{equation}\label{eqn:lorenz cond}
\partial^\alpha A_\alpha = 0.
\end{equation}
Given this we can rewrite the equations for the gauge potential $A_\alpha$ as the wave equation
\begin{equation*}
\Box A_\alpha=\partial^\beta\partial_\beta A_\alpha =  - J_\alpha.
\end{equation*}
This does not completely characterize $A$, in that we can add any one-form of the form $d\psi$ to $A$, where $\psi$ is a solution to the wave equation, and still recover the Lorenz gauge. The Lorenz gauge propagates through time, so if the solution $A_\alpha$ satisfies the Lorenz gauge condition \eqref{eqn:lorenz cond} at $t=0$, then \eqref{eqn:lorenz cond} in fact holds for all times.

    Our goal is to give a precise description of the asymptotic behaviour of the scalar (complex) field $\phi$, and the gauge $A_\mu$, evolved from data at $t=0$. Some care has to be taken however, as the data for the gauge $A_\alpha$ must satisfy the constraints
        $$ \p_t A_0 = \p^j A_j, \qquad \p^j \p_t A_j - \Delta A_0 = J_0 $$
    which arise from the Lorenz gauge condition, and the equation for $F_{0 \beta}$. It particular, it suffices to impose the data
        \begin{equation}\label{eqn:data MKG lorenz}
             \big( \phi(0), D_0 \phi(0) \big) = (\phi_0, \dot{\phi}_0), \qquad \big(A_j(0), \p_t A_j(0) \big) = \big( a_j, \dot{a}_j\big) \text{ for } j=1, 2, 3.
        \end{equation}
    The data for the temporal component of the gauge $(A_0, \p_t A_0)(0)$, can then be constructed via the constraint equations. We give the details of this argument in Section \ref{sec:compatible data} below.

        To study the system \eref{eq:waveequation}, we introduce a \emph{null frame}.
The first two members of this are the null generators of forward and
backward light cones which we define respectively as:
\begin{align*}
        L \ &= \ \partial_t + \partial_r \ , &\bL \ &= \
    \partial_t - \partial_r
    \  \label{L_bL}
\end{align*}
where $r=|x|$. To obtain a basis for vector fields on $\RR^{1+3}$, it only remains to define derivatives in the angular directions.
This can be done in an identical fashion on each time slice $\{t = const\}$
so we only need to define things on $\mathbb{R}^3$. If we let $\{e^0_B\}_{B=1,2}$ denote a local  orthonormal frame for the unit sphere
in $\mathbb{R}^3$, then for each value of the radial variable
we can by extension define:
\begin{equation*}
        S_B \ = \ \frac{1}{r}\, e^0_B \ . \label{ang_frame}
\end{equation*}
Thus  $\{S_B\}_{B=1,2}$ forms an orthonormal basis on each
sphere $\{r=const.\}$, for each fixed time slice. The potential $A_\mu$ can be expressed in the frame $\{L, \bL, S_1, S_2\}$, for instance writing $L = L^\mu \p_\mu$, we have
        $$ A_L = L^\mu A_\mu = A_0 + \omega^j A_j, \qquad A_{\bL} = \bL^\mu A_\mu = A_0 - \omega^j A_j.$$
For later use, we note that the coefficients of the frame can also be raised and lowered using the metric $m$, thus $L_\mu = m_{\mu \nu} L^\nu$, and $L_\mu = L_{\mu}(\omega)$ is a function of $\omega \in \mathbb{S}^2$ only.

\subsection{Results} In previous work on Maxwell Klein Gordon systems first global existence for similar systems was shown by Eardly-Moncrief \cite{EM82a,EM82b} with  refinement by Klainerman-Machedon \cite{KM94}. Later decay estimates were shown in Lindblad-Sterbentz \cite{LS06} after  preliminary results in Shu \cite{S91} and Psarelli \cite{P99}. Recently extensions of \cite{LS06} were given in
Yang \cite{Y15}, Bieri-Miao-Shahshahani \cite{BMS17}, Kauffman \cite{K18} and Klainerman-Wang-Yang \cite{KWY18}.

Here we show that the Maxwell-Klein-Gordon system in the Lorenz gauge satisfy the {\it weak null condition} of Lindblad-Rodnianski \cite{LR05,LR10} and we give the detailed asymptotics of the field and the potential. These asymptotics have two parts, one wave like along outgoing light cones at null infinity, and one homogeneous inside the light cone at time like infinity. Here the charge play a crucial role imposing an oscillating factor in the asymptotic system for the field. Similar results have previously been shown for Einstein's equations in wave coordinates in Lindblad \cite{L17}.

Our results rely on the decay estimates obtained in \cite{LS06}, which require certain natural smallness conditions on the data. To this end, we
define the weighted Sobolev spaces:
\begin{align*}
        \lp{T}{H^{k,s_0}(\RR^3)}^2 \ &= \ \sum_{|I|\leqslant k}\
    \int_{\RR}\ (1 + r^2)^{s_0 + |I|} \ |\nabla_x^I\, T|^2 \
    dx \ . \label{tensor_initial_sob}
\end{align*}
We will assume that for some $1/2<s_0<3/2$ and sufficiently small $\epsilon>0$ our initial data satisfy
\begin{equation}\label{eq:inidatanorm}
\| (a_1, a_2, a_3) \|_{H^{k+1, s_0 -1}} + \| (\dot{a}_1, \dot{a}_2, \dot{a}_3) \|_{H^{k, s_0}} + \| \phi_0 \|_{H^{k+1, s_0-1}} + \| \dot{\phi}_0 \|_{H^{k, s_0}} \les \epsilon.
\end{equation}
Given $s_0$ let  $s$ and $\gamma$ be any numbers such that
\begin{equation}\label{eq:sandgammacond}
\frac{1}{2}<s<1,\qquad 0<\gamma < \frac{3}{2}-s\qquad\text{and}\qquad
{s_0^\prime} = s+\gamma<s_0.
\end{equation}
Our main results are the following:

\subsubsection{Asymptotics at null infinity along the light cones}
\begin{Theorem}\label{thm:asymp in exterior}
Let $k\g 7$ with $k\in\mathbb{N}$ and assume that \eqref{eq:inidatanorm} and \eqref{eq:sandgammacond} hold.
Then provided that $\epsilon$ is sufficiently small, if $(A_\mu, \phi)$ is the unique global solution to \eqref{eq:waveequation} with constraint \eqref{eqn:lorenz cond} and data \eqref{eqn:data MKG lorenz}, for any $q\in \RR$, $\omega \in \mathbb{S}^2$ the limits
    \begin{equation*}
        \Phi_0(q, \omega)= \lim_{t \rightarrow \infty} \big( r e^{i \frac{1}{4\pi} {\mathbf{q}}   \ln (1+r)}\phi\big)\big( t, (t + q) \omega\big),
    \end{equation*}
and
    \begin{equation*}
 \mathcal{A}_{S_B}(q,\omega)=\lim_{t\to\infty}
(r {A}_{S_B})\big(t,(q+t)\omega\big), \qquad \qquad \mc{A}_{\ull}^0(q, \omega)= \lim_{t \to \infty} ( r A^{mod}_{\ull})\big( t, (t + q) \omega \big)
    \end{equation*}
exist, where we define
$$ A^{mod}_{\ull}(t,r\omega) = A_{\ull}\big( t, r\omega\big)
            -\frac{1}{2r} \int_{r-t}^\infty \mc{J}_{\ull}(\eta, \omega) \ln\Big( \frac{\eta + t+r}{\eta + t-r}\Big) d\eta $$
and
    $$ \mc{J}_{\ull}(q,\omega) = - 2\Im\Big( \Phi_0(q, \omega) \overline{\p_q \Phi_0(q, \omega)} \Big). $$
Moreover, for $r=|x|>\frac{t}{2}$ and $\omega = \frac{x}{|x|}$ we have
    \begin{align*}
      \big|  ( e^{ i {\frac{1}{4\pi} \mathbf{q}}   \ln{(1+r)}} r \phi) (t,x)  -  \Phi_0(r-t, \omega)\Big| &\lesa
       \epsilon  \lr{t+r}^{\frac{1}{2} - (s+\gamma)} +   \epsilon\lr{t+r}^{-1} \lr{t-r}^{\frac{3}{2}-s} S^0(t,r) \ind_{\{t>r\}}.
    \end{align*}
and
    \begin{align*}
       \big| (r A_L)(t,x) - {\frac{1}{4\pi} \mathbf{q}}   \big| + &\big| (r A_{S_B})(t,x) -  \mc{A}_{S_B}(r-t, \omega)\big|+ \big| ( r A^{mod}_{\ull})(t,r\omega) - \mc{A}_{\ull}^0(r-t, \omega) \big| \\
         &\lesa \epsilon \lr{t+r}^{\frac{1}{2}-(s+\gamma)} + \epsilon^2 \frac{\lr{t-r}}{\lr{t+r}} S^0\,  \ind_{\{t>r\}}.
    \end{align*}
    Here $\ind_{\{t>r\}}$ is the characteristic function for the set where $t>r$ and
    $$S^0(t,r) = \frac{t+r}{r} \ln\Big( \frac{\lr{t+r}}{\lr{t-r}}\Big). $$
\end{Theorem}

Note that for free solutions to the wave equation, $\Box u = 0$, on $\RR^{1+3}$, the limit $\lim_{t \to \infty} (r u)(t, (q+t)\omega)$ exists, and the difference decays like $r^{-2}$, i.e. one order better than the decay of $u$ in the exterior $t< \frac{1}{2}|x|$. Hence the previous theorem shows that at null infinity,  $A_{S_B}$ behaves like a free wave, $A_L$ converges to the charge $\mathbf{q}$, and when compared to the evolution of a free wave, the scalar field $\phi$ has phase correction at null infinity. On the other hand the bad component of gauge $A_{\underline{L}}$ only behaves like a free wave after subtracting off a term with a log growth. In particular, $A_{\underline{L}}$ has a log loss of decay when compared to the free wave equation.

\subsubsection{Asymptotics at timelike infinity in the interior}
We have the following asymptotics at time like infinity:
\begin{Theorem}\label{thm:asymp in interior} Let $k\g 7$ with $k\in\mathbb{N}$ and assume that \eqref{eq:inidatanorm} and \eqref{eq:sandgammacond} hold.
Given $y \in \mathbb{B}^3=\{y\in \mathbb{R}^3; \,|y|<1\}$, we have the limit
\[
\lim_{t \to \infty}tA_\mu(t, ty ) = \mc{K}_\mu(y),\quad\text{where}\quad \mc{K}_\mu(y)=\frac{1}{4\pi}\int_{-\infty}^\infty\int_{\mathbb{S}^2}\frac{\mc{J}_\mu(q, \omega)}{1-\lr{y, \omega}}\, dS(\omega) \, dq \,,
\]
where $\mc{J}_\mu$ is the asymptotic source term
\begin{equation*}
\mc{J}_\mu(q,\omega) = L_\mu(\omega)\Im\left(\Phi_0(q, \omega)\overline{\partial_q\Phi_0(q, \omega)}\right).
\end{equation*}
Here $L_\mu(\omega)=m_{\mu\nu} L^\nu(\omega)$, where
$L=\pa_t+\pa_r=L^\nu(\omega)\pa_\nu$.

Additionally, we have the following bound on the difference when $|x|<t$
\begin{equation}\label{eq:interassymest}
\left|\,tA_\mu(t, x) - \mc{K}_\mu(x/t)\right| \lesssim  \epsilon\big|\,t-|x|\big|^{\,1-2s} S^0+ \epsilon\big|\,t-|x|\big|^{\,1/2-s-\gamma}.
\end{equation}
\end{Theorem}

\begin{Remark} The second term in the right of \eqref{eq:interassymest} corresponds to a solution of a homogeneous wave equation that has same kind of asymptotics at null infinity as given in Theorem \ref{thm:asymp in exterior}, compare Lindblad-Schlue \cite{LS17}.
\end{Remark}

\subsubsection{Improved interior decay for the field $\phi$} Our method starts from the decay estimates in \cite{LS06}, see Theorem \ref{thm:global stability}, but as byproduct we also get slightly improved interior decay for the field $\phi$ when compared to \cite{LS06}, see Proposition \ref{prop:improvedinteriordecayforfield}.

\subsection{Outline of paper}

The article is organised as follows. In Section \ref{sec:weak null} we show that the M-K-G system satisfies the weak null condition of Lindblad-Rodnianski \cite{LR05,LR10}, and sketch the proof of Theorem \ref{thm:asymp in exterior}. In Section \ref{sec:gauge invariant est} we recall the gauge invariant estimates of \cite{LS06}, and deduce decay of the current $J_\mu$ and the field $\phi$. To apply the decay estimates in \cite{LS06}, we need to show that we can construct a compatible data set from the data in Theorem \ref{thm:asymp in exterior}, this is done in \ref{sec:compatible data}. In Section \ref{sec:decay for A} we prove decay estimates for the potential $A_\mu$, while in Sections \ref{sec:asymptotics for tang A} and \ref{sec:asymp for phi} we prove the exterior asymptotics for the tangential components of $A_\mu$, and the field $\phi$. The exterior asymptotics for the non-tangential component of $A_\mu$ is contained in Section \ref{sec:asymp for non-tang A}. Finally in Section \ref{sec:interior asymp}, we prove the interior asymptotics contained in Theorem \ref{thm:asymp in interior}. The appendix, Section \ref{sec:appendix}, contains a slight refinement of the radial estimate from \cite{L17}.

\section{The asymptotic system for M-K-G and the weak null condition}\label{sec:weak null}

\subsection{The asymptotic system at null infinity} The asymptotic system introduced by H{\"o}rmander \cite{H97} is obtained by plugging the expansions
\begin{equation*}
A_\mu(t,r\omega)\sim \mc{A}_\mu(r-t,\omega,\ln{r})/r, \qquad \phi(t,r\omega)\sim \Phi(r-t,\omega,\ln{r})/r
\end{equation*}
into the equations and neglecting angular derivatives and derivatives tangential to the light cones
\begin{equation*}
\aligned \Box\, \phi\! &=\!r^{-1}(\pa_t+\pa_r)(\pa_r-\pa_t)
(_{\!}r\phi_{\!})+r^{-2}\times \text{angular
derivatives},\\
\pa_\mu&=\tfrac{1}{2}\hat{\omega}_\mu(\pa_r-\pa_t)+\text{tangential
derivatives},\qquad \hat{\omega}=(-1,\omega),
\endaligned
\end{equation*}
 in which case one gets the asymptotic system
\begin{equation*}
(\pa_t+\pa_r)(\pa_r-\pa_t) \mc{A}_\mu\sim\frac{1}{r} \frac{ L_\mu}{2} \Im(\Phi \overline{(\pa_r-\pa_t) \Phi})
\end{equation*}
and if we write
\[
D^\alpha D_\alpha \phi = \Box\, \phi + 2iA^\alpha \partial_\alpha\phi - A^\alpha \! A_\alpha\phi=0.
\]
we also get
\begin{equation*}
(\partial_t+\partial_r)(\partial_r-\partial_t)\Phi\sim -i \frac{1}{r}\mc{A}_L (\partial_r-\partial_t)\Phi.
\end{equation*}
If we also introduce the independent variables
$$
q=r-t,\qquad\text{and}\qquad s=\ln{r}
$$
and neglect the lower order term when $(\pa_t-\pa_r)$ is falling on $\ln{r}$ we get the asymptotic system
\begin{align}\label{eq:AAssSyst}
\pa_s\pa_q \mc{A}_\mu&= \frac{ L_\mu}{2} \Im(\Phi \overline{\pa_q\Phi})\\
\pa_s\pa_q \Phi &= -i \mc{A}_L \pa_q \Phi.\label{eq:PhiAssSyst}
\end{align}
The system satisfies the {\it weak null condition} of Lindblad-Rodnianski \cite{LR05,LR10} if the asymptotic system above has a
global solution that does not grow too much in $s$. We will show that this is the case below.

\subsubsection{The asymptotic system for the good components of $A$}
Given this framework, we can look at the asymptotic system for components of $A$. Since the frame commutes with the radial part of the wave operator above we have
\[
\pa_s\pa_q \mc{A}_\mc{T} =0,\qquad \mc{T} \in \{L,S_1,S_2\},
\]
from which it follows that $ \mc{A}_{\mc{T}} $ is a function of $r-t$ and $\omega$ only.
In fact we will be able to prove that
\begin{equation*}
\mc{A}_L=\frac{1}{4\pi} \mathbf{q}
\end{equation*}
where $\mathbf{q}$ is the charge, see section \ref{section:charge}.

\subsubsection{The asymptotic system for the field}
We are now ready to look at the asymptotic system
where all other terms involve nice derivatives, which we can assume do not factor into the system. Asymptotically, we note that if we take the real and imaginary parts of $\Phi$, and noting that $\mc{A}_L=\frac{1}{4\pi} \mathbf{q}$ is real and constant,  we get by \eqref{eq:PhiAssSyst} a function $\Phi_0$ such that
\begin{equation*}
\Phi(q,\omega,s)= e^{-i\frac{1}{4\pi} \mathbf{q} s} \Phi_0(q,\omega)
\end{equation*}

\subsubsection{The asymptotic system for the bad component of the potential}
By \eqref{eq:AAssSyst} we have
\begin{equation*}
\pa_s\pa_q  \mc{A}_\ull=\mc{J}_L,\quad\text{where}\quad \mc{J}_L=\Im(\Phi \overline{\pa_q  \Phi})=\Im(\Phi_0 \overline{\pa_q  \Phi_0})
\end{equation*}
from which it follows that
\begin{equation*}
\mc{A}_\ull(q,\omega,s)=s\int_q^\infty \mc{J}_L(\rho,\omega)\, d\rho+\mc{A}_\ull^0(q,\omega).
\end{equation*}

\subsection{The asymptotics at time like infinity} To precisely analyse the asymptotics of $A_\alpha$, we let $A^2_\alpha$ denote the inhomogeneous component of $A_\alpha$. Thus $A^2_\alpha$ is the solution to
\begin{equation*}
\Box A^2_\alpha= - J_\alpha,
\end{equation*}
with vanishing initial data. From the asymptotics in the exterior region $t<2 r$, we have
\begin{equation*}
J_\alpha(t,r\omega) \sim \frac{1}{r^2}\mc{J}_\alpha(r-t,\omega),\quad \text{where}\quad
|\mc{J}_\alpha(q,\omega)|\lesssim \langle q\rangle^{-2s}\langle q_+\rangle^{-2\gamma}.
\end{equation*}
Since $s>1/2$, the asymptotic current $\mc{J}_\alpha(q,\omega)$ is integrable with respect to $q$ and concentrated close to the light cone $q=0$. Therefore for an observer far away from the light cone  $t-r \gg 0$ it looks
like the total mass of the source comes from the light cone $q=0$. Therefore is some rescaled variable
$y=x/t$ we have as measures
\begin{equation*}
 \frac{1}{r^2}\mc{J}_\alpha(r-t,\omega)\sim \frac{1}{r^2}\delta(t-r)\mc{M}_\alpha(\omega),\quad\text{where}\quad  \mc{M}_\alpha(\omega)= \int_{-\infty}^{+\infty} \mc{J}_\alpha(q,\omega) dq,
\end{equation*}
where $\delta(s)$ is the delta function, see Lindblad \cite{L90}. Moreover by \cite{L90}, with $\mc{K}_\alpha$ as in Theorem \ref{thm:asymp in interior}, we have in the sense of distributions
\begin{equation*}
\Box\Big( \frac{1}{t}\mc{K}_\alpha(x/t)\Big)= \frac{1}{r^2}\delta(t-r)\mc{M}_\alpha(\omega).
\end{equation*}

\subsection{The charge contribution}\label{section:charge}
To understand how the charge may effect the evolution, note that the data for the gauge $A_\mu$, $(a_\mu, \dot{a}_\mu)$, satisfies the constraint
    $$  \Delta a_0 = \p^j \dot{a}_j - J_0, \qquad \dot{a}_0 = \p_j a_j.   $$
In particular, if $\int_{\RR^3} J_0 \not = 0$, then we can only expect the decay $a_0 \approx \frac{1}{r}$ as $r \rightarrow \infty$. This causes problems as to bootstrap decay for the gauge fields $A_\mu$, we require additional decay in the exterior region $r > t$. The way to proceed, following Lindblad-Sterbenz \cite{LS06}, is to subtract off the worst decaying component of the data. More precisely, we define the modified gauge field $A^1_\mu$ as
    $$ A^1_\mu = A_\mu - \delta_{\mu 0} \chi( r-t) \frac{\mathbf{q}}{4 \pi r} .$$
where $\chi$ is a smooth cutoff such that $\chi(s) = 1$ if $s\g 1$, and $\chi(s) = 0$ if $s<{1}/{2}$. Then
    $$ \Box A^1_\mu = -J_\mu$$
so the modified fields satisfy the same equation as $A$ and a computation shows that we have improved decay
    $$ A^1_0(0, x) \approx r^{-2}$$
as $r\rightarrow \infty$ (the remaining data has the same decay as the original decay). In particular, in the following we use the decay bounds
    \begin{equation*}
|Z^I A_\delta^1| \lesssim \langle r\rangle^{-s-1/2-\gamma},\qquad\text{when}\quad  t=0
    \end{equation*}
for $Z \in \{ \p_\mu, \Omega_{\mu \nu}, S\}$ where $\Omega_{\mu \nu} = x_\mu \p_\nu - x_\nu \p_\mu$, $S= x^\mu \p_\mu$, which follow from the finiteness of the weighted Sobolev norms assumed in Theorem \ref{thm:asymp in exterior}.

\subsection{Key steps in the proof of Theorem \ref{thm:asymp in exterior}}
We now briefly outline the key steps in the proof of Theorem \ref{thm:asymp in exterior}.

\textbf{Step 1:} The first step is to recall the pointwise decay estimates obtained in Lindblad-Sterbenz \cite{LS06} (see (8.10) there)  which give
    \begin{equation*}\label{eq:Jestiamteintro}
   | Z^I J_\mu| \lesssim \epsilon^2\lr{t+r}^{-2}\lr{t-r}^{-2s} \lr{(r-t)_+}^{-2\gamma},
\end{equation*}
with the tangential components satisfying the stronger bound
\begin{equation*}\label{eq:goodJestiamteintro}
\begin{split}
|Z^K J_L| &\lesssim \epsilon^2 \lr{t+r}^{-\frac{5}{2}-s}\lr{t-r}^{\frac{1}{2}-s} \lr{(r-t)_+}^{-2\gamma},\\
 |Z^K J_{S_B}| &\lesssim \epsilon^2 \lr{t+r}^{-3}\lr{t-r}^{1-2s} \lr{(r-t)_+}^{-2\gamma}
\end{split}
\end{equation*}
for $Z \in \{ \p_\mu, \Omega_{\mu \nu}, S\}$ where $\Omega_{\mu \nu} = x_\mu \p_\nu - x_\nu \p_\mu$, $S= x^\mu \p_\mu$ (here $x^0 = t$) and the parameters $s, \gamma$ satisfy
    $$ s+\gamma < {3}/{2}, \qquad \gamma>0, \qquad s>{1}/{2}.$$
Together with the positivity of the fundamental solution to wave equation, and the identity
    \begin{equation}\label{eqn:pointwise radial wave}
    \begin{split}
        (ru)(t,x) = \ind_{\{t<r\}} (ru)\big( 0, (r-t)\omega &\big) + \frac{1}{2} \int_{|t-r|}^{t+r} ( \p_t + \p_r)(ru)(0, \xi \omega) d\xi
            \\
            &+ \frac{1}{4} \int_{|t-r|}^{t+r} \int_{-\xi}^{t-r} (\p_t - \p_r)(\p_t + \p_r)(ru)\big( \tfrac{1}{2}(\xi +\eta), \tfrac{1}{2}(\xi-\eta) \omega \big) d\eta d\xi
    \end{split}
    \end{equation}
(here $ x = r \omega$, $\omega \in \mathbb{S}^2$, $r>0$) we then deduce the weak decay bounds
    $$
    |Z^K A^1_\mu(t, x)| \lesa \epsilon^2\lr{t+r}^{-1} \lr{(r-t)_+}^{1-2(s+\gamma)} S^0(t,r)+ \epsilon\langle t+r\rangle^{-1} \langle r-t\rangle^{\frac{1}{2} -(s+\gamma)}
    $$
where
    $$
    S^0(t,r) = \frac{t+r}{r} \ln\Big( \frac{\lr{t+r}}{\lr{t-r}}\Big)\lesssim\frac{1}{\epsilon}
    \frac{\lr{t+r}^\epsilon }{\lr{t-r}^{\epsilon}},
    $$
and thus we have a loss of decay when compared to the free wave equation. The estimate on $S^0(t,r)$ for small $\epsilon$ comes from the inequality
\[
\ln\left(\frac{\lr{t+r}}{\lr{t-r}}\right) = \frac12\ln\left(1 + \frac{4tr}{1+(t-r)^2}\right) \leq \frac{2tr}{1+(t-r)^2}
\]
when $r < t/2$, and the fundamental theorem of calculus when $r>t/2$.

 Once we have the decay of all components of $A_\mu$, we can improve this decay for the tangential components of the field $A_\mu$ by using \eref{eq:goodJestiamte} although this requires a loss of derivatives (which is not important here).

\textbf{Step 2:} The second step is deduce the asymptotics for the best component of the gauge $A_L$. More precisely, a computation shows that $A_L$ satisfies
    $$ L( r \ull( rA_L) ) = r^2 J_L + \Delta_{\mathbb{S}^2} ( A^1_L)- L(r A^1_{\ull}). $$
Exploiting the weak decay bounds, together with the improved decay for $J_L$, the right hand side has sufficient decay, and by integrating along $t\pm r$ we obtain
    $$ | r A_L - {\frac{1}{4\pi} \mathbf{q}}   | \lesa \epsilon\lr{t+r}^{\frac12 - (s+\gamma)} + \epsilon\frac{\lr{t-r}}{\lr{t+r}}S^0(t, r)\ind_{\{ t > r \}}.  $$
To obtain the asymptotics for the remaining tangential component, $A_{S_B}$, a slightly easier argument suffices.

\textbf{Step 3:} The third step is to deduce the behaviour of $\phi$ as $t+r \rightarrow \infty$. This follows by observing that we have an identity of the form
    $$ L \ull\big( \exp{ (i {\frac{1}{4\pi} \mathbf{q}}   \ln{r})} r \phi \big) = G$$
with $G$ having sufficient decay to deduce that, via \eref{eqn:pointwise radial wave}, the existence of a limit along light cones.

\textbf{Step 4:} The final step is to exploit the refined asymptotics in Lindblad \cite{L17} to obtain the asymptotic behaviour of the non-tangential component $A_{\ull}$, and the limits for the current $J_\mu$ which are required in the asymptotics in the interior region $|x|<t$.

\section{The Gauge invariant estimates}\label{sec:gauge invariant est}
The initial data for the system \eqref{eq:waveequation} can be written in the form:
\begin{subequations}
\begin{align}
        F_{0i}(0)  \ &= \ E_i  \ ,
    &\sF_{0i}(0)  \ &= \  H_i  \ , \\
    \phi(0) \ &= \ \phi_0 \ ,
    &D_t \phi\, (0) \ &= \dot{\phi}_0.
\end{align}
\end{subequations}
In the above formulas $\sF$ denotes the Hodge dual of $F$ which
is given by the expression $\sF_{\alpha\beta} = \frac{1}{2}
\in_{\alpha\beta}^{\ \ \ \gamma\delta}F_{\gamma\delta}$. Here
$\in_{\alpha\beta\gamma\delta}$ denotes the volume form on Minkowski
space.
Now, from the form of the system \eqref{eq:waveequation} it is easy to see that
this initial data cannot be specified freely. It must also satisfy the compatibility
conditions:
\begin{align}
        \nabla^i E_i \ &= \ \Im [ \phi_0 \, \overline{\dot{\phi}_0
    }] \ ,
    &\nabla^i H_i \ &= \ 0 .
    \label{compat_conds}
\end{align}
The first equation comes from expanding $\partial^iE_i = \partial^\alpha F_{0\alpha} - \partial^0F_{00}$ and noting that the latter term is 0 from antisymmetry of $F$. The second equation comes from expanding and applying the identity
\[
(dF)_{ijk} = 0.
\]

We will call a data set $(E,H, \phi_0 , \dot{\phi}_0 )$ which
satisfies \eqref{compat_conds} \emph{admissible}.

To state the main result from \cite{LS06}, we first define the covariant weighted Sobolev spaces $H^{k,s}_{cov}$ using the norm
    \begin{equation*}
        \lp{\psi}{H_{cov}^{k,s_0}(\RR^3)}^2 \ = \ \sum_{|I|\leqslant k}\
    \int_{\RR}\ (1 + r^2)^{s_0 + |I|}\ |\underline{D}^I \psi |^2
    dx  \label{scalar_initial_sob}
    \end{equation*}
where $\underline{D} = \nabla + i \underline{a}$ with $\underline{a} = (a_1, a_2, a_3)$ denoting the spatial components of the gauge at $t=0$.

\begin{Theorem}[Global Stability of CSF Equations {\cite{LS06}}]\label{thm:global stability}
Let $k\geq 2$ with $k\in\mathbb{N}$, and let ${s_0^\prime} = s+\gamma$
be given such that ${s_0^\prime} < \frac{3}{2}$, $\gamma>0$, and
$\frac{1}{2}<s\leq 1$. Let $(E,H,\phi_0,\dot{\phi}_0)$ be an admissible
initial data set, and define the \emph{charge} to be the value:
\begin{equation*}
        {\mathbf{q}} \ =\ \int_{\RR^3}\ \Im [\phi_0\, \overline{\dot{\phi}_0}] dx\ . \label{the_charge}
\end{equation*}
Then there exists a universal constant $\mathcal{E}_{k,s,\gamma}$,
which depends only on the parameters $k,s,\gamma$, such that if
$(E,H,\phi_0,\dot{\phi}_0)$ is an admissible initial data set which
satisfies the smallness condition:
\begin{equation*}
        \lp{E^{df}}{H^{k,{s_0^\prime}}(\RR^3)} + \lp{H}{H^{k,{s_0^\prime}}(\RR^3)} +
    \lp{\underline{D}\phi_0 }{H^{k,{s_0^\prime}}_{cov}(\RR^3)} +
    \lp{\dot{\phi}_0}{H^{k,{s_0^\prime}}_{cov}(\RR^3)} \ \leqslant \ \mathcal{E}_{k,s,\gamma}
    \ , \label{initial_smallness}
\end{equation*}
where $E = E^{df} + E^{cf}$ is the Hodge decomposition of $E$ into
its divergence free and curl free components (resp.), then there
exists a (unique) global solution to the system of equations
\eqref{eq:waveequation} with this initial data set such that if
$\{L,\bL,S_A\}$ denotes a standard spherical null frame, then the following
point-wise properties of this solution holds:
\begin{subequations}\label{F_peeling}
\begin{align}
        |\alpha| \ &\lesssim \ \mathcal{E}_{k,s,\gamma}\tp^{-s
        - \frac{3}{2}} \lr{(r-t)_+}^{-\gamma} \ ,\label{F_peeling1}\\
   |\balpha| \ &\lesssim \ \mathcal{E}_{k,s,\gamma } \tp^{-1}
        \tm^{-s-\frac{1}{2}}  \lr{(r-t)_+}^{-\gamma} \ , \label{F_peeling2}\\
    |\rho| \ &\lesssim \ {\mathbf{q}} \, r^{-2} \,
         \chi_{ 1 < t < r} \ + \ \mathcal{E}_{k,s,\gamma }
         \tp^{-1-s}\tm^{-\frac{1}{2}}\lr{(r-t)_+}^{-\gamma}  \ , \label{F_peeling3}\\
     |\sigma| \ &\lesssim \ \mathcal{E}_{k,s,\gamma }
     \tp^{-1-s}\tm^{-\frac{1}{2}} \lr{(r-t)_+}^{-\gamma} \ , \label{F_peeling4}
\end{align}
\end{subequations}
and:
\begin{subequations}\label{phi_peeling}
\begin{align}
         |\widetilde{D}_L\phi|  &\lesssim
    \mathcal{E}_{k,s,\gamma } \tp^{-s -
         \frac{3}{2}}\lr{(r-t)_+}^{-\gamma},\label{phi_peeling1}\\
    |D_{\bL}\phi| &\lesssim   \mathcal{E}_{k,s,\gamma }
    \tp^{-1}
        \tm^{-s-\frac{1}{2}}  \lr{(r-t)_+}^{-\gamma}, \label{phi_peeling2}\\
    |\sD\phi| &\lesssim\mathcal{E}_{k,s,\gamma }
     \tp^{-1-s}\tm^{-\frac{1}{2}} \lr{(r-t)_+}^{-\gamma} \!\! ,\label{phi_peeling3}\\
     |\phi| &\lesssim  \mathcal{E}_{k,s,\gamma }\tp^{-1}\tm^{-s+\frac{1}{2}}\cdot
     \lr{(r-t)_+}^{-\gamma} \!\!
     . \label{phi_peeling4}
\end{align}
\end{subequations}
Here we have set:
\begin{align}
        |\widetilde{D}_L\phi|^2 \ &= \
        |\frac{1}{r}D_L(r\phi)|^2\, \chi_{1 < t < 2r} +
    |D_L\phi|^2 \, \chi_{r < \frac{1}{2}t} \ ,
    &|\sD\phi|^2 \ &= \  \delta^{AB} D_{S_A}\phi \overline{D_{S_B}\phi} \
        , \notag
\end{align}
and $\underline{D}$ denotes the spatial part of the connection $D$.
Also, $(\alpha,\balpha,\rho,\sigma)$ denotes the components of the
null decomposition \eqref{null_decomp} of $F_{\alpha\beta}$.
\end{Theorem}
 Here we used the notation $\alpha=\alpha(F)$, $\balpha=\balpha(F)$, $\rho=\rho(F)$ and
 $\sigma=\sigma(F)$ where
\begin{subequations}\label{null_decomp}
\begin{align}
        \alpha_A \ &= \ F_{LA} \ , &\balpha_A \ &= \ F_{\bL A}\ ,
        \label{null_decomp1}\\
    \rho \ &= \ \frac{1}{2}F_{\bL L} \ ,
    &\sigma \ &= \ \frac{1}{2} \in^{AB} F_{AB}
    \ . \label{null_decomp2}
\end{align}
\end{subequations}

The argument in \cite{LS06} in fact also decay bounds for derivatives of $F$ and $\phi$. More precisely, by \cite[Proposition 7.1, Theorem 8.1]{LS06}, we have the following decay (peeling) properties similar to
\eqref{F_peeling}--\eqref{phi_peeling} for the higher derivatives of
$(F,\phi)$ assuming that $2 < k$. In particular we have for $|I|\leq k-2$ and vector fields in the inhomogeneous Lorentz algebra $Z \in \{ \p_\mu, \Omega_{\mu \nu}, S\}$ (here $\Omega_{\mu \nu} = x_\mu \p_\nu - x_\nu \p_\mu$, $S= x^\mu \p_\mu$, and $x^0 = t$) the estimates
\begin{subequations}\label{phi_peelinghigher}
\begin{align}
         |\widetilde{D}_L D_Z^I\phi|  &\lesssim
    \mathcal{E}_{k,s,\gamma } \tp^{-s -
         \frac{3}{2}}\lr{(r-t)_+}^{-\gamma}  ,\label{phi_peeling1higher}\\
    |D_{\bL}D_Z^I\phi|  &\lesssim  \mathcal{E}_{k,s,\gamma }
    \tp^{-1}
        \tm^{-s-\frac{1}{2}} \lr{(r-t)_+}^{-\gamma} , \label{phi_peeling2higher}\\
    |\sD D_Z^I\phi|  &\lesssim  \mathcal{E}_{k,s,\gamma }
     \tp^{-1-s}\tm^{-\frac{1}{2}} \lr{(r-t)_+}^{-\gamma}  ,\label{phi_peeling3higher}\\
     |D_Z^I\phi| &\lesssim \mathcal{E}_{k,s,\gamma } \tp^{-1}\tm^{-s+\frac{1}{2}}
     \lr{(r-t)_+}^{-\gamma}
     , \label{phi_peeling4higher}
\end{align}
\end{subequations}
and
\begin{subequations}\label{F_peelinghigher}
\begin{align}
        |\alpha(\Lie{Z}^I F)| &\lesssim {\mathbf{q}}\,  r^{-3} \tm\,
         \chi_{ 1 < t < r}  +\mathcal{E}_{k,s,\gamma} \tp^{-s
        - \frac{3}{2}} \lr{(r-t)_+}^{-\gamma},\\
    \ \ \ \ \ \ \ \ \ \ \ \ \ \ \
    |\balpha(\Lie{Z}^I F)| &\lesssim  {\mathbf{q}}\,  r^{-2} \,
         \chi_{ 1 < t < r}  +\mathcal{E}_{k,s,\gamma } \tp^{-1}
        \tm^{-s-\frac{1}{2}} \lr{(r-t)_+}^{-\gamma} , \label{F_peeling1higher}\\
    |\rho(\Lie{Z}^I F)|  &\lesssim  {\mathbf{q}}\,  r^{-2} \,
         \chi_{ 1 < t < r}  +  \mathcal{E}_{k,s,\gamma }
         \tp^{-1-s}\tm^{-\frac{1}{2}}\lr{(r-t)_+}^{-\gamma}  , \label{F_peeling2higher}\\
     |\sigma(\Lie{Z}^I F)|  &\lesssim {\mathbf{q}}\,  r^{-2} \,
         \chi_{ 1 < t < r}  +\mathcal{E}_{k,s,\gamma }
     \tp^{-1-s}\tm^{-\frac{1}{2}}\lr{(r-t)_+}^{-\gamma}  , \label{F_peeling3higher}
\end{align}
\end{subequations}

For the sake of brevity, we use the consequent inequalities
\begin{subequations}\label{F_peelingreduced}
\begin{align}
        |\alpha(\Lie{Z}^I F)|  &\lesssim  \mathcal{E}_{k,s,\gamma} \tp^{-s
        - \frac{3}{2}}\tm^{s-\frac12},\\
    \ \ \ \ \ \ \ \ \ \ \ \ \ \ \
    |\balpha(\Lie{Z}^I F)|  &\lesssim  \mathcal{E}_{k,s,\gamma } \tp^{-1}
        \tm^{-1}  , \label{F_peeling1reduced}\\
    |\rho(\Lie{Z}^I F)| &\lesssim  \mathcal{E}_{k,s,\gamma }
         \tp^{-1-s}\tm^{s-1}  , \label{F_peeling2reduced}\\
     |\sigma(\Lie{Z}^I F)|  &\lesssim \mathcal{E}_{k,s,\gamma }
     \tp^{-1-s}\tm^{s-1}. \label{F_peeling3reduced}
\end{align}
\end{subequations}

We can now show estimates on derivatives of components of $J$ as follows:
\begin{Proposition}\label{prop:decay of current} With notation and  assumptions as in Theorem
\ref{thm:global stability} we have for $|I|\les k-2$
    \begin{equation}\label{eq:Jestiamte}
   | Z^I J_\mu| \lesssim \epsilon^2\lr{t+r}^{-2}\lr{t-r}^{-2s} \lr{(r-t)_+}^{-2\gamma},
\end{equation}
with the tangential components satisfying the following stronger bound in the region $r > \frac{t+1}{2}$:
\begin{equation}\label{eq:goodJestiamte}
\begin{split}
|Z^I J_L| + |(\Lie{Z}^I J)_L| &\lesssim \epsilon^2 \lr{t+r}^{-\frac{5}{2}-s}\lr{t-r}^{\frac{1}{2}-s} \lr{(r-t)_+}^{-2\gamma},\\
 |Z^I J_{S_B}| + |(\Lie{Z}^I J)_{S_B}|&\lesssim \epsilon^2 \lr{t+r}^{-3}\lr{t-r}^{1-2s} \lr{(r-t)_+}^{-2\gamma}
\end{split}
\end{equation}
for $Z \in \{ \p_\mu, \Omega_{\mu \nu}, S\}$ where $\Omega_{\mu \nu} = x_\mu \p_\nu - x_\nu \p_\mu$, $S= x^\mu \p_\mu$ (here $x^0 = t$).

\end{Proposition}
\begin{proof}
We prove this in the region $r > t/2$, as in the far interior the proof is easier. We first show that components of Lie derivatives $\Lie{Z}^IJ$ satisfy the same bounds. We have the identity
\begin{align*}
[\Lie{Z}\left(\phi\overline{D\psi}\right)]_\alpha &= Z^\beta \p_\beta ( \phi \overline{D_\alpha \psi} ) + \p_\alpha Z^\beta \phi \overline{D_\beta \psi} \\
    &= D_Z\phi\overline{D_\alpha\psi} + \phi \overline{D_\alpha D_Z\psi} + iZ^\beta F_{\alpha\beta}\phi\overline{\psi},
\end{align*}
which follows from the relations
\[
\partial_\alpha(\phi\overline\psi) = D_\alpha\phi\overline\psi + \phi\overline{D_\alpha\psi}, \qquad D_\alpha D_\beta \psi = D_\beta D_\alpha \psi + i F_{\alpha \beta} \psi.
\]
We iterate this and take the imaginary part to get the identity
\begin{equation}\label{eq:LieJ}
(\Lie{Z}^IJ)_\alpha = \sum_{|J|+|K| = |I|} c^{I}_{JK}\Im\left( D_Z^J\phi \overline{D_\alpha D_Z^K\phi}\right) + \sum_{|J| + |K| + |L| + 1 \leq |I|}c^{Ia}_{JKL}\Re\left(D_Z^J\phi \overline{D_Z^K\phi}Z_a^\beta (\Lie{Z}^L F)_{\alpha\beta}\right).
\end{equation}

The second term on the right in \eqref{eq:LieJ} satisfies our bounds when contracted with any null vector, which follows almost immediately from
\begin{subequations}\label{F_comms}
\begin{align}
\sum_{|J| + |K| + 1 \leq |I|} |\Lie{Z}^J F(Z, L)||D_Z^K\phi|  &\lesssim \mathcal{E}_{k,s,\gamma}^2\tp^{-s-\frac32} \lr{(r-t)_+}^{-\gamma}, \\
\sum_{|J| + |K| + 1 \leq |I|} |\Lie{Z}^J F(Z, \underline{L})||D_Z^K\phi| &\lesssim \mathcal{E}_{k,s,\gamma}^2 \tp^{-1}\tm^{-s-\frac12} \lr{(r-t)_+}^{-\gamma} , \\
\sum_{|J| + |K| + 1 \leq |I|} |\Lie{Z}^J F(Z, S_A)||D_Z^K\phi| &\lesssim \mathcal{E}_{k,s,\gamma}^2 \tp^{-1-s}\tm^{-\frac12} \lr{(r-t)_+}^{-\gamma} .
\end{align}
\end{subequations}
The estimates \eqref{F_comms} in turn follow directly from decomposition of the field $Z$ which is contracted with $F$ into its null components, combined with the estimates \eqref{F_peelingreduced} and \eqref{phi_peeling2higher}.

We now consider the first term on the right in \eqref{eq:LieJ}. By our iteration, $c_I^{JK} = c_I^{KJ}$. We can therefore write
\begin{equation*}
\sum_{|J|+|K| = |I|} c^{I}_{JK}\Im\left( D_Z^J\phi \overline{D_\alpha D_Z^K\phi}\right) = \frac12\sum_{|J|+|K| = |I|} c^{I}_{JK}\Im\left( D_Z^J\phi \overline{D_\alpha D_Z^K\phi} + D_Z^K\phi \overline{D_\alpha D_Z^J\phi}\right)
\end{equation*}

We can replace the right hand side with
\[
\frac12\sum_{|J|+|K| = |I|} c_{I}^{JK}\Im\left( D_Z^J\phi \overline{\frac{D_\alpha (rD_Z^K\phi)}{r}} + D_Z^K\phi \overline{\frac{D_\alpha (rD_Z^J\phi)}{r}}\right),
\]
noting that the difference is
\[
\frac12\sum_{|J|+|K| = |I|} c_{I}^{JK}\Im\left( D_Z^J\phi \frac{\partial_\alpha(r)}{r}\overline{D_Z^K\phi} + D_Z^K\phi \frac{\partial_\alpha(r)}{r}\overline{D_Z^J\phi}\right),
\]
which is the imaginary part of a real quantity. Our bounds on components of $\Lie{Z}^I J$ then follow from  $L^\infty$ estimates in Theorem \ref{thm:global stability}. Consequently, we see that the required bounds \eref{eq:Jestiamte} and \eqref{eq:goodJestiamte} hold for the Lie derivatives $(\mc{L}_Z^I J)_T$. To prove the bound \eqref{eq:Jestiamte} for $Z^I J_\mu$, we first observe that we have the identity
\begin{equation}
\label{ZJ_decomp}
Z^IA_\mathcal{U} = \sum_{|K|+|L| = |I|}c^{I}_{KL}(\Lie{Z}^K A)_\alpha (\Lie{Z}^L \mathcal{U})^\alpha
\end{equation}
for any vector field $\mc{U}$, and any $1$-form $A_\mu$. Since, $\Lie{Z}^L (\partial_\mu) = c^{L}_{\mu\nu}\partial_\nu$, the estimate \eqref{eq:Jestiamte} follows immediately from the bounds on the Lie derivatives. For the derivatives of the tangential components, we apply the following lemma.
\end{proof}

\begin{lemma}\label{lem:lie derivative vs derivative}
Let $A_\mu$ be a $1$-form. Then for any multi-index $I$ we have
    \begin{align*}
|Z^I A_L| &\lesssim  \sum_{|K| \leq |I|}|(\Lie{Z}^K A)_L| + \frac{\lr{t-r}}{\lr{t+r}}|(\Lie{Z}^K A)_{S}| + \frac{\lr{t-r}^2}{\lr{t+r}^2}|(\Lie{Z}^K A)_{\underline{L}}|, \\
|Z^I A_S| &\lesssim  \sum_{|K| \leq |I|}|(\Lie{Z}^K A)_L| + |(\Lie{Z}^K A)_{S}| + \frac{\lr{t-r}}{\lr{t+r}}|(\Lie{Z}^K A)_{\underline{L}}|.
\end{align*}
\end{lemma}
\begin{proof}
We start by observing that
\begin{subequations}
\begin{align*}
(\Lie{Z}^I L)^\alpha &= f_1L^\alpha + \sum_i f_2^i S_i^\alpha + f_3 \underline{L}^\alpha \\
(\Lie{Z}^I S_j)^\alpha &= g_{j1}L^\alpha + \sum_i g_{j2}^i S_i^\alpha + g_{j3} \underline{L}^\alpha
\end{align*}
\end{subequations}
where $f_1, g_{j1}, g_{j2}^i$ can all be written in the form
\[
\sum_{\substack{i+j+k \leq n \leq |I|}}f_{ijkn}^I(\omega) \frac{(t-r)^it^j(t+r)^k}{r^n},
\]
$f_2, g_{j3}$ can be written in the form
\[
\sum_{\substack{i+j+k \leq n\leq |I| \\ j+k \leq n-1}}f_{ijkn}^I(\omega) \frac{(t-r)^it^j(t+r)^k}{r^n},
\]
and $f_3$ can be written in the form
\[
\sum_{\substack{i+j+k \leq n \leq |I|\\ j+k \leq n-2}}f_{ijkn}^I(\omega) \frac{(t-r)^it^j(t+r)^k}{r^n}.
\]
This follows a relatively straightforward but tedious inductive proof. It follows that
\begin{subequations}
\begin{align*}
|f_1|+|g_{j1}|+|g_{j2}^i| &\lesssim_I 1 \\
|f_2|+|g_{j3}|&\lesssim_I \frac{\lr{t-r}}{\lr{t+r}} \\
|f_3| &\lesssim_I \frac{\lr{t-r}^2}{\lr{t+r}^2}.
\end{align*}
\end{subequations}
We can combine these estimates with the identity \eqref{ZJ_decomp} to get
\begin{subequations}
\begin{align*}
|Z^I A_L| &\lesssim  \sum_{|K| \leq |I|}|(\Lie{Z}^K A)_L| + \frac{\lr{t-r}}{\lr{t+r}}|(\Lie{Z}^K A)_{S}| + \frac{\lr{t-r}^2}{\lr{t+r}^2}|(\Lie{Z}^K A)_{\underline{L}}|, \\
|Z^I A_S| &\lesssim  \sum_{|K| \leq |I|}|(\Lie{Z}^K A)_L| + |(\Lie{Z}^K A)_{S}| + \frac{\lr{t-r}}{\lr{t+r}}|(\Lie{Z}^K A)_{\underline{L}}|
\end{align*}
\end{subequations}
as required.
\end{proof}

\section{Compatible data for $A$ and the charge contribution}\label{sec:compatible data}

In this section our goal is to show that we can construct an admissible data set for $F$ (in the sense of Theorem \ref{thm:global stability}) from the data $(a_j, \dot{a}_j)$ given in the statement of Theorem \ref{thm:asymp in exterior}, which satisfies the required smallness condition. We start by defining the data for $A_0$ at $t=0$, $(a_0, \dot{a}_0)$, in terms of the data $(a_j, \dot{a}_j)$ and $(\phi_0, \dot{\phi}_0)$
    \begin{equation}\label{eq:laplacea0}
        \Delta a_0 = \p^j \dot{a}_j - \Im(\phi_0 \overline{\dot{\phi_0}}), \qquad \qquad \dot{a}_0 = \p^j a_j.
    \end{equation}
The first equation in \eref{eq:laplacea0} arises from the equation $\p^\beta F_{0\beta} = J_0=\Im [\phi\, \overline{D_0\phi}] $ evaluated at $t=0$, while the second is the Lorenz gauge condition at $t=0$. The fact that $a_0$ is a solution to an elliptic equation is responsible for the lack of decay of $a_0$ in $r$, and leads to significant difficulties in the global analysis for the MKG system. To understand how the charge may effect $a_0$, note that if
$$
{\mathbf{q}} = \int_{\RR^3} J_0 dx\neq 0,
$$
 then we can only expect the solution of \eqref{eq:laplacea0} to decay $a_0 \approx \frac{1}{r}$ as $r \rightarrow \infty$. This causes problems as to bootstrap decay for the gauge fields $A_\mu$, we require additional decay in the exterior region $r > t$. The way to proceed, following \cite{LS06}, is to subtract off the worst decay component of the data. More precisely, we define the modified gauge field $A^1_\mu$ as
    $$ A^1_\mu = A_\mu - \delta_{\mu 0} \chi( r-t) \frac{{\mathbf{q}}  }{4 \pi r}.$$
where $\chi$ is a smooth cutoff such that $\chi(s) = 1$ if $s\g 1$, and $\chi(s) = 0$ if $s<\frac{1}{2}$. Then
    $$ \Box A^1_\mu = -J_\mu$$
thus the modified fields satisfy the same equation as $A_\mu$,
and a computation shows that we have the improved decay
    $$ A^1_0(0, x) \approx \frac{1}{r^2}$$
as $r\rightarrow \infty$ which is sufficient to close a bootstrap argument.

Given the full data set $(a_\mu, \dot{a}_\mu)$ for $A_\mu$, we construct an admissible data set $(E, H)$ for $F$ by defining
    \begin{equation}\label{eq:atoE}
        E = \dot{\underline{a}} - \nabla a_0, \qquad H = \nabla \times \underline{a}
    \end{equation}
where $\underline{a} = (a_1, a_2, a_3)$ and $\underline{\dot{a}} = (\dot{a}_1, \dot{a}_2, \dot{a}_3)$ denotes the spatial components of the data for the gauge $A$. Clearly since $H$ is a curl, it is divergence free. Moreover, in view of \eref{eq:laplacea0}, we see that $\p^j E_j = \Im[\phi_0 \overline{\dot{\phi}_0}]$. Hence the data $(E, H)$ satisfies the compatibility conditions \eref{compat_conds}. The smallness condition is a consequence of the following.

\begin{lemma}\label{lem:smallness of energy}
Let $k\g 1$ with $k\in\mathbb{N}$, and suppose that for some $\frac{1}{2}<s_0 < \frac{3}{2}$ we have
    \begin{equation}\label{eq:aphidata}
     \| \underline{a} \|_{H^{k+1, {s_0}-1}} + \| \underline{\dot{a}} \|_{H^{k, {s_0}}}
     + \| \phi_0 \|_{H^{k+1, {s_0} -1}} + \| \dot{\phi}_0 \|_{H^{k, {s_0}}} \les \epsilon.
     \end{equation}
    Define $(a_0, \dot{a}_0)$ and $(E, H)$ as in \eref{eq:laplacea0} and \eref{eq:atoE}.
Then for any ${s_0^\prime} <s_0$ we have
    \begin{equation}\label{eq:EHphidata} \| E^{df} \|_{H^{k, {s_0^\prime}}}
    + \| H \|_{H^{k, {s_0^\prime}}} + \big\| a_0 - \tfrac{1}{4\pi r} \chi(r) \mathbf{q} \big\|_{H^{k+1, {s_0^\prime} -1}} + \| \dot{a}_0 \|_{H^{k, {s_0^\prime}}} + \| \phi_0\|_{H^{k+1, {s_0^\prime}-1}_{cov} } + \| \dot{\phi}_0 \|_{H^{k, {s_0^\prime}}_{cov}}
        \lesa \epsilon,
    \end{equation}
    where $\chi(r)$ is a smooth function such that $\chi(r)=1$ for $r>1$, and $\chi = 0$ for $r<\frac{1}{2}$.
\end{lemma}
\begin{proof}
The bounds for $H$ and $\dot{a}_0$ follow directly from \eref{eq:laplacea0} and \eref{eq:atoE}. To bound $E^{df}$ and $a_0$, we first observe that integrating by parts easily gives for $\alpha>\frac{1}{2}$
    $$ \int_{\RR^3} (1+r^2)^\alpha |\nabla \psi |^2 dx \lesa \int_{\RR^3} ( 1 + r^2)^{\alpha - 1} |\psi|^2 dx + \int_{\RR^3} ( 1+ r^2)^{\alpha+1}  |\Delta \psi|^2 dx $$
and
    $$ \sum_{i, j} \int_{\RR^3} (1+r)^\alpha |\p_j \p_i  \psi |^2 dx \lesa \int_{\RR^3} ( 1 + r)^{\alpha - 1} |\nabla \psi|^2 dx + \int_{\RR^3} ( 1+ r)^{\alpha}  |\Delta \psi|^2 dx. $$
In particular, after $k$ applications, we see
    $$ \big\| a_0 - \frac{1}{4\pi r} \chi(r) \mathbf{q} \big\|_{H^{k+1, {s_0^\prime} -1}}^2 \lesa \int_{\RR^3} (1 + r^2)^{{s_0^\prime}} \big| \nabla \big(a_0 - \frac{1}{4\pi r} \chi(r) \mathbf{q}\big) \big|^2 dx + \| \dot{a}_j  - \Im(\phi_0 \overline{\dot{\phi}_0}) \|_{H^{k, {s_0^\prime}}}$$
and after writing $E^{df} = \p_j \dot{a}_j - \nabla \Delta^{-1} \p^j \dot{a}_j$ (since gradient $\nabla a_0$ is curl free, and, as we will see, is decaying faster than $r^{-1}$)
    $$ \| E^{df} \|_{H^{s, s + \gamma}}^2 \lesa \int_{\RR^3} ( 1 + r^2)^{{s_0^\prime}} | \Delta^{-1} \p^j \dot{a}_j |^2 dx + \| \dot{a}_j \|_{H^{k, {s_0^\prime}}}.$$
Therefore, using the decay of $(\phi_0, \dot{\phi}_0)$, (and hence $\Im(\phi_0 \overline{\dot{\phi}_0})$) it is enough to show that
    $$ \int_{\RR^3} ( 1+ r^2)^{{s_0^\prime} -1} \big| \nabla \big(\psi + \frac{1}{4\pi r} \widetilde{\mathbf{q}} \big)\big|^2 dx \lesa \int_{\RR^3} ( 1 + r^2)^{{s_0^\prime}} | \Delta \psi|^2 dx $$
with $\widetilde{\mathbf{q}} = \int_{\RR^3} \Delta \psi dx $. But this follows from the appendix in \cite{LS06}. More precisely, from  \cite[Lemma 10.1]{LS06} we have
\begin{equation*}
        \int_{\RR^3} \ r^{2({s_0^\prime}-1)}\, \Big| \nabla
        \left(\psi + \frac{\widetilde{\mathbf{q}}}{4\pi
        r}\right)\Big|^2 \ dx \ \lesssim \ \lp{r^{{s_0^\prime}-1}\,
         \Delta \psi }{L^\frac{6}{5}}^2 , \label{weighted_grad_sob}
\end{equation*}
and we conclude by observing that
\begin{equation*}
\lp{r^{{s_0^\prime}-1} \,
         \triangle \psi }{L^\frac{6}{5}}\leq \lp{\langle r\rangle^{{s_0}}\,
         \triangle \psi }{L^2} \lp{\langle r\rangle^{{s_0^\prime} -{s_0} -1}\,}{L^{3}}\leq C_\mu \lp{\langle r\rangle^{{s_0^\prime}}\,
         \triangle \psi }{L^2} ,\qquad\text{if}\quad s_0>{s_0^\prime}.
\end{equation*}
Finally, to conclude that we may bound $\| \phi_0\|_{H^{k+1, {s_0^\prime}-1}_{cov}} + \| \dot{\phi}_0 \|_{H^{k, {s_0^\prime}}_{cov}}$ by the weighted Sobolev space without the covariant derivatives, simply follows from the weighted Sobolev lemma below, together with an induction argument.
\end{proof}

To deduce the pointwise decay of $A^1$, we use the following.

\begin{Proposition} With notation and assumptions as in Lemma \ref{lem:smallness of energy}
we have for $|I|\leq k-2$
\begin{equation}\label{eq:A1initial}
|Z^I A_\delta^1| + |Z^I \phi| \lesssim \langle r\rangle^{-{s_0^\prime}-1/2},\qquad\text{when}\quad  t=0.
    \end{equation}
\end{Proposition}
The proof follows from the $L^2$ estimates in Lemma \ref{lem:smallness of energy} together with the weighted Sobolev lemma below. (Together with that we get the second and higher time derivatives from the equation.)

\begin{lemma} Let $\omega=x/r\in \mathbb{S}^2$.  Then for $N'\geq N+3$
\begin{equation*}\label{eq:decayassumptioninfty}
\sum_{|\alpha|+k\leq N} \sup_{r\geq 0}
\sup_{\omega\in\mathbb{S}^2} \big|(\langle\, r\rangle\pa_r)^k\pa_\omega^\alpha
F(r,\omega)\big|\langle r\rangle^{{s_0^\prime}+1/2}\,
 \lesssim\!\!\sum_{|\alpha|+k\leq N'} \!\!\!\!
 \Big(\int_{0}^\infty \!\!\!\!\int_{\mathbb{S}^2} \!\!\!\big|
 (\langle\, r\rangle\pa_r)^k\pa_\omega^\alpha F(r,\omega)\big|^2\langle r\rangle^{2({s_0^\prime} -1)} d S(\omega) r^2\! dr\Big)^{\!1\!/2}\!\!\!.
\end{equation*}
\end{lemma}
\begin{proof}
  Consider the case $N'=0$. Then for $p=2({s_0^\prime}-1)$,
  \begin{equation*}
    |F(r,\omega)|\leq |\int_r^\infty \partial_r F d r|\leq \Bigl(\int_{\mathbb{R}}\langle r\rangle^p (\langle r\rangle\partial_r F)^2 r^2 dr \Bigr)^\frac{1}{2}\Bigl(\int_r^\infty\frac{d r}{\langle r\rangle^{4+p}}\Bigr)^\frac{1}{2}
    \lesssim
     \Bigl(\int_{\mathbb{R}}\langle r\rangle^p (\langle r\rangle\partial_r F)^2 r^2 dr \Bigr)^\frac{1}{2}\langle r\rangle^{\frac{-p-3}{2}}
  \end{equation*}
which proves the inequality if we also estimate the maximum over $\omega \in \mathbb{S}^2$ by the Sobolev norm $H^2(\mathbb{S}^2)$ because $(p+3)/2=s_0+1/2$.
\end{proof}

\section{Decay for all components of $A_\mu$}\label{sec:decay for A}

Suppose we have data as in Theorem \ref{thm:asymp in exterior}. Then Lemma \ref{lem:smallness of energy} together with Theorem \ref{thm:global stability} implies that we have a global solution $(A_\alpha, \phi)$ with a charge $J_\alpha$ satisfying the decay conditions in Proposition \ref{prop:decay of current}. In the following sections, we show how the decay of the charge implies the pointwise decay of the potential, as well as precise asymptotics in the exterior region $2t < r$. In particular, we always assume that we have a global solution $(A_\mu, \phi)$ satisfying the decay conditions in Theorem \ref{thm:global stability} and Proposition \ref{prop:decay of current}. \\

We start by proving the weak decay estimate for all components of $A_\mu$.

\begin{Proposition}\label{prop:weak decay of A}
With notation and assumptions as in Lemma \ref{lem:smallness of energy}
we have for $|I|\leq k-3$
        \begin{equation}\label{eq:weakA1bound}
            \big| Z^I A_\mu^1\big|\lesa \epsilon^2 S^0(t,r)\langle t+r\rangle^{-1}
            \langle (r-t)_+\rangle^{1 -2{s_0^\prime}}
            + \epsilon\langle t+r\rangle^{-1}
             \langle r-t\rangle^{\frac{1}{2} -{s_0^\prime}}
             \lesssim \epsilon \langle t+r\rangle^{-1+\epsilon}\langle t-r\rangle^{-\epsilon}
            \langle (r-t)_+\rangle^{1/2 -{s_0^\prime}} .
        \end{equation}
    Here
    \beqs S^0(t,r)
    =\frac{t}{r\!\!}\,
    \ln{\!\Big(\frac{\langle \,t\!+r\,\rangle}{\langle \,t\!-r\rangle}\Big)}
    \les
    \frac{1}{\varepsilon}
    \Big(\frac{\langle\, t\!+r \rangle}{\langle \,t\!-r\rangle}\Big)^{\varepsilon}\!\! .
    \label{eq:loggrowth}
    \eqs
\end{Proposition}
\begin{proof} In order to get the bound for $|I|=0$ we decompose
 $ A_\mu^1= A_\mu^2+ A_\mu^0$, where $A_\mu^0$ is the solution to
the homogenous problem $\Box A_\mu^0=0$ with the same initial data as $A^1_\mu$, and
$A_\mu^2$ is the solution of $\Box A^2_\mu= - J_\mu$ with vanishing initial data.
Using Lemma \ref{lem:inhomwaveeqdecay} and  \eref{eq:Jestiamte} gives that $A_\mu^2$
is bounded by the first term in the right hand side of \eqref{eq:weakA1bound}.
Using Lemma \ref{lem:homwaveeqdecay} and \eqref{eq:A1initial} (for $|I|\leq 1$)  gives that $A_\mu^0$
is bounded by the second term in the right hand side of \eqref{eq:weakA1bound}.
\eqref{eq:weakA1bound} for $|I|>1$ follows from first commuting with the vector fields $Z$
and then decomposing into a homogenous part and an inhomogeneous part and estimating
each part as above.
\end{proof}

The next step is to use the improved decay of the tangential components of the current \eref{eq:goodJestiamte}, to deduce improved bounds for the tangential components of the gauge $A_\mu$.

\begin{Proposition}\label{prop:decay for tangential comp of A}
With notation and assumptions as in Lemma \ref{lem:smallness of energy}
we have for $|I|\leq k-5$
\begin{align*}
\big|({\mathcal L}_Z^I {A}^1)_{T}\big|
&\lesa \epsilon^2 \langle t\!+\!r\rangle^{-1}
\langle t\!-\!r\rangle^{1-2s} \langle (r\!-\!t)_+\rangle^{-2\gamma}
+ \epsilon\langle t+r\rangle^{-1} \langle r-t\rangle^{\frac{1}{2}-{s_0^\prime}}
\lesssim \epsilon \langle t+r\rangle^{-1} \langle (r-t)_+\rangle^{1/2-{s_0^\prime}}
\end{align*}
for $T=L, S_1, S_2$.
\end{Proposition}

\begin{proof}
We start with the case $|I|=0$. As in the previous proof, it is enough to bound the inhomogeneous component $A^2_\mu$. Similarly to the proof of Lemma 12 in  \cite{L17}, contraction with the null frame does not commute with $\Box$, which leads to a commutator term. However the commutator term involves angular derivatives which can be absorbed on the righthand side of the equation by using the weak decay bounds obtained in the previous proposition. More precisely, we first observe that we can write
    $$ L \ull ( rA_T^2) =  -r T^\mu \Box A_\mu^2 +  \frac{1}{r} T^\mu \Delta_{\mathbb{S}^2} A_\mu^2 = r J_T  + \frac{1}{r} T^\mu \Delta_{\mathbb{S}^2} A_\mu^2$$
 with $\Delta_{\mathbb{S}^2} = \sum \Omega_{ij}^2$. Noting that $r^{-1}|\Omega_{ij}  \psi| =|(\omega^i\pa_{\!j}-\omega^{j}\pa_i)  \psi|\les \langle t\!+r\rangle^{-1}\! \sum_{|I|=1}|Z^I \psi|$, we see that
\beqs
r^{-2}|\triangle_{\mathbb{S}^2} A^2_\mu |\les \langle r\rangle^{-1} \langle t+r\rangle^{-1}{\sum}_{|I|\leq 2}|Z^I A^2_\mu|,
\eqs
 and hence by Proposition \ref{prop:weak decay of A} and \eref{eq:goodJestiamte}
 we obtain for any $s+\gamma={s_0^\prime}$ with $1/2<s<1$ and $\gamma>0$
    $$
    r^{-1}|L \ull ( rA_T^2)|
     \lesa \epsilon^2 \lr{t\!+\!r}^{-3}\lr{t\!-\!r}^{1-2s} \lr{(r\!-\!t)_+}^{-2\gamma}
     + \epsilon\lr{r}^{-1} \lr{t\!+\!r}^{\epsilon -2} \lr{(r\!-\!t)_+}^{1/2-{s_0^\prime}}
     \lr{t-r}^{-\epsilon}.
     $$
 The required decay for $|I|=0$ then follows from the identity \eqref{eqn:pointwise radial wave}, together with a short computation. Alternatively, we can argue using Lemma \ref{lem:inhomwaveeqdecay}.  More precisely, take $F(t,r) = \sup_{\omega \in \mathbb{S}^3}  r^{-1}| L \ull( r A_T^2)(t, r \omega)|$ and let $\psi$ solve $\Box \psi = F$ with vanishing data. Then by the positivity of the fundamental solution, $|A_T^2| \les |\psi|$, and hence result for $|I|=0$ follows by applying Lemma \ref{lem:inhomwaveeqdecay} to $\psi$. To prove the required bound for $|I|>0$, we simply observe that for $Z \in \{\p_\mu, \Omega_{\mu \nu}, S\}$, the Lie derivatives commute with $\Box$, i.e. we have $\Box \mc{L}_Z A = \mc{L}_Z(\Box A) + c \Box A$ with $c \in \{0, 1\}$.  Hence we can repeat the argument used in the case $|I|=0$.
 \end{proof}

\section{The asymptotics for the tangential components of $A_\mu$}\label{sec:asymptotics for tang A}
In this section we deduce the asymptotic behaviour of the gauge potential $A_\mu$ along light cones. Our first result is the following.

\begin{Proposition}[Exterior Asymptotics for angular components]
With notation and assumptions as in Lemma \ref{lem:smallness of energy} with $k\geq 5$
the following hold.  The limit
\beqs
 \mathcal{A}_{S_B}(q,\omega)=\lim_{t \to\infty}
(r {A}_{S_B})\big(t,(q+t)\omega\big)
 \eqs
 exists and moreover we have the bound
 $$
 \big| r {A}_{S_B}(t,x)-\mathcal{A}_{S_B}(r-t,\omega)\big|
 \lesa  \epsilon \lr{t+r}^{\frac{1}{2} - {s_0^\prime}} + \epsilon^2 \lr{t+r}^{-1} \lr{t-r} S^0 \ind_{t>r}.
 $$
\end{Proposition}
\begin{proof} Let
    $$ F = r \Box  A_{S_B} - \frac{1}{r} (S_B)^j \Delta_{\mathbb{S}^2} A_j. $$
An application of the identity \eref{eqn:pointwise radial wave} implies that for $t_1 > t_0>0$ and $ q > -t_0 $ we have
    \begin{align*}
      (rA)_{S_B}\big( t_1, (t_1 + q) \omega\big)& - (rA)_{S_B}\big(t_0, (t_0 + q)\omega \big) \\
        &= \frac{1}{2} \int_{2 t_0 + q}^{2 t_1 + q} (\p_t + \p_r)( r A_{S_B})(0,  \xi \omega) d\xi + \frac{1}{4} \int_{2t_0 + q}^{2t_1 + q} \int_{-\xi}^{-q} F\big( \tfrac{1}{2}(\xi + \eta), \tfrac{1}{2}(\xi-\eta) \big) d\eta d\xi.
    \end{align*}
The decay assumption on the data implies that
    $$ \int_{2 t_0 + q}^{2 t_1 + q}\big| (\p_t + \p_r)( r A_{S_B})(0,  \xi \omega)
    \big| d\xi \lesa \lr{2t_0+q}^{\frac{1}{2} - {s_0^\prime}}.$$
On the other hand, since $A^1_{S_B} = A_{S_B}$, \eref{eq:goodJestiamte} and
Proposition \ref{prop:weak decay of A} imply that for any $s+\gamma={s_0^\prime}$ with $1/2<s<1$ and $\gamma>0$
    $$ \big|F\big( \tfrac{1}{2}(\xi + \eta), \tfrac{1}{2}(\xi-\eta) \big)\big| \lesa  \epsilon^2 \lr{\xi}^{-2} \lr{\eta}^{1-2s} \lr{(-\eta)_+}^{-2\gamma} + \epsilon \lr{\xi}^{-2} \lr{(-\eta)_+}^{1-2{s_0^\prime}} S^0 +  \epsilon \lr{\xi}^{-2} \lr{\eta}^{\frac{1}{2}-{s_0^\prime}}\big)     $$
and hence a computation gives
   \begin{align*} \int_{2t_0 + q}^{2t_1 + q} \int_{-\xi}^{-q} \big| F\big( \tfrac{1}{2}(\xi +& \eta), \tfrac{1}{2}(\xi-\eta) \big)\big| d\eta d\xi \lesa \epsilon \lr{2t_0+q}^{\frac{1}{2} - {s_0^\prime}} + \epsilon^2 \lr{2t_0+q}^{-1} \lr{q} S^0 \ind_{q<0}
   \end{align*}
where we used the fact that for $0<a<1$
    \begin{equation}\label{eq:logintegral}
    \int_{-(t+r)}^{0}  \Big(1+\ln\Big(\frac{t+r}{|\eta|}\Big)\Big)\frac{d\eta}{|\eta|^{a}}
    =(t+r)^{1-a}  \int_{-1}^{0}  \Big(1+\ln\Big(\frac{1}{|\eta|}\Big)\Big)\frac{d\eta}{|\eta|^{a}}.
    \end{equation}
Therefore we conclude that
    \begin{align*}
      \big|(rA_{S_B})\big( t_1, (t_1 + q) \omega\big) - (rA_{S_B})\big(t_0, (t_0 + q)\omega \big)\big|
            &\lesa  \epsilon \lr{2t_0+q}^{\frac{1}{2} - {s_0^\prime}} + \epsilon^2 \lr{2t_0+q}^{-1} \lr{q} S^0 \ind_{q<0}
    \end{align*}
and hence the required limit exists and satisfies the claimed bound.
\end{proof}

To deduce the limit for the $A_L$ component requires more care in commuting the frame with the wave operator $\Box$. In particular, we need to exploit the key identity
    \begin{equation}\label{eqn:L commuted with box}
        L\big( r \ull( r A_L) \big) +  L(  r A_{\ull}^1 )=  r^2 J_L +  \Delta_{\mathbb{S}^2} A_L^1
    \end{equation}
which follows by observing that the Lorenz gauge condition implies that commuting the frame with the angular derivatives gives
    $$\Delta_{\mathbb{S}^2} A_L = L^\mu \Delta_{\mathbb{S}^2} A_\mu  + L( r A_{\ull}) + \ull (r A_L). $$
The identity \eqref{eqn:L commuted with box} is then a consequence of the fact that $L( r A_{\ull}) = L( r A_{\ull}^1)$, $\Delta_{\mathbb{S}^2} A_L = \Delta_{\mathbb{S}^2} A_L^1$, and the standard radial decomposition of $\Box$. Integrating \eref{eqn:L commuted with box} along characteristics $t\pm r$ and applying the decay bounds obtained earlier leads to the following.

\begin{Proposition}[Exterior Asymptotics for $A_L$] With notation and assumptions as in Lemma \ref{lem:smallness of energy} with $k\g 7$
the following hold. The limit
$$
\lim_{ t \to\infty}
(r {A}_{L})\big(t,(t+q)\omega\big) = \frac{{\mathbf{q}}}{4\pi}
 $$
 exists and moreover we have the bound
 \begin{equation}\label{eqn:AL asymp bound}
 \Big| r A_L(t,x) - \frac{1}{4\pi} {\mathbf{q}}   \Big|
 \lesa  \epsilon \lr{t+r}^{\frac{1}{2} - {s_0^\prime}}
 +\epsilon \frac{\lr{t-r}}{\lr{t+r}} S^0(t,r)\ind_{\{ t>r \}}.
 \end{equation}
\end{Proposition}
\begin{proof}
Integrating the identity \eref{eqn:L commuted with box} in $t+r$ to the initial data (or the axis $r=0$ if $t>r$) gives
     \begin{align*} \big(r &\ull( r A_L)+ r A_{\ull}^1 \big)(t,x) \\
        &= \ind_{\{ t<r \}} \big( r \ull ( r A_L) +  r A_{\ull}^1\big)\big(0, (r-t)\omega\big) + \frac{1}{2} \int_{|t-r|}^{t+r} \Big( r^2 J_L -  \Delta_{\mathbb{S}^2} A_L^1\Big)\big( \tfrac{1}{2}(\xi + t-r), \tfrac{1}{2}(\xi - t + r)\big) d\xi .
     \end{align*}
Note that
    $$ \big|(r\ull( r A_L))\big( 0, (r-t)\omega\big)  + (rA_{\ull}^1)
    \big(0, (r-t)\omega\big)\big|\lesa \epsilon\lr{r-t}^{\frac{1}{2} - {s_0^\prime}}$$
and from \eref{eq:goodJestiamte}, Lemma \ref{lem:lie derivative vs derivative}, and Proposition \ref{prop:decay for tangential comp of A}
    \begin{align*}\big|\big( r^2 J_L -  \Delta_{\mathbb{S}^2} A_L^1\big)(t,x)\big|
                &\lesa  \epsilon^2 \lr{t+r}^{-1} \lr{(r-t)_+}^{1 - 2{s_0^\prime}} + \epsilon \lr{t+r}^{-1} \lr{(r-t)_+}^{\frac{1}{2}-{s_0^\prime}}\\
                &\leq \epsilon \lr{t+r}^{-1} \lr{(r-t)_+}^{\frac{1}{2}-{s_0^\prime}}.
    \end{align*}
    Hence
\begin{equation*}
 \int_{|t-r|}^{t+r} \Big| r^2 J_L -  \Delta_{\mathbb{S}^2} A_L^1\Big|\big( \tfrac{1}{2}(\xi + t-r), \tfrac{1}{2}(\xi - t + r)\big) d\xi
 \lesssim \epsilon\ln\Big(\frac{\lr{t+r}}{\lr{t-r}}\Big)  \lr{(r-t)_+}^{\frac{1}{2}-{s_0^\prime}}
\end{equation*}
where we used \eref{eq:logintegral}. Therefore
    $$ \big|\ull( r A_{L})(t,x)\big| \lesa \epsilon\Big(1+\ln\Big(\frac{\lr{t+r}}{\lr{t-r}}\Big)\Big)\lr{t+r}^{-1}  \lr{(r-t)_+}^{\frac{1}{2}-{s_0^\prime}}.
    $$
If we now integrate along $t-r$, we obtain
    $$ ( r A_L)(t,x) = (r A_L)\big(0, (t+r)\omega \big) + \frac{1}{2} \int_{-(t+r)}^{t-r} \ull( r A_L)\big( \tfrac{1}{2}( t+r + \eta), \tfrac{1}{2}(t+r - \eta)\big) d\eta. $$
After noting that Lemma \ref{lem:smallness of energy} gives
    $$\Big|(r A_L)\big(0, (t+r)\omega \big) - \frac{1}{4\pi}
     \mathbf{q}  \Big| \lesa  \epsilon \lr{t+r}^{\frac{1}{2} - {s_0^\prime}}$$
after combining the above bounds we finally deduce that
    $$ \Big| r A_L(t,x) - \frac{1}{4\pi} {\mathbf{q}}   \Big| \lesa \epsilon \lr{t+r}^{\frac{1}{2} - {s_0^\prime}}  + \epsilon \frac{\lr{t-r}}{\lr{t+r}} \Big(1+\ln\Big(\frac{\lr{t+r}}{\lr{t-r}}\Big)\Big)\ind_{\{ t>r \}} . $$
\end{proof}

\section{Asymptotics of the field $\phi$}\label{sec:asymp for phi}

We now turn to the asymptotic behaviour of $\phi$ in the exterior region $r>t$.
We start by giving the decay estimates for $\phi$.
\begin{Proposition}\label{prop:exter decay for phi}
Let $|\slashed{\partial} \phi|^2=\delta^{AB}\partial_{S_A}\phi \, \,
\overline{\partial_{S_A}\phi }$.
With notation and assumptions as in Lemma \ref{lem:smallness of energy} we have for
$s+\gamma={s_0^\prime}$ and $1/2<s<1$ and $\gamma>0$
         \begin{equation}\label{eqn:decay of phi}
         {\sum}_{|I|\leq 2}|Z^I\phi|+\lr{t+r} \big(|L \phi|+|\slashed{\partial}\phi|\big)
         +\lr{t-r}\big(|\ull \phi | +|\partial \phi|\big) \lesa  \epsilon\lr{t+r}^{-1} \lr{t-r}^{\frac{1}{2}-s} \lr{(r-t)_+}^{-\gamma},\end{equation}
 \end{Proposition}
\begin{proof}
  The required bounds follow  from \eqref{phi_peelinghigher} together with Proposition 
  \ref{prop:weak decay of A} and Proposition \ref{prop:decay for tangential comp of A}. First we note that
  $|Z^\alpha A_\alpha|\lesssim \langle t-r\rangle |A|+\langle t+r\rangle \big(|A_L|+|A_{S_1}|+|A_{S_2}|\big)$, which proves the case $|I|=1$. Secondly we note that
  $\mc{L}_{Z} (X^\alpha A_\alpha)=  (\mc{L}_{Z} X^\alpha) A_\alpha+X^\alpha (\mc{L}_{Z} A_\alpha)$,
  where $\mc{L}_{Z} X^\alpha=[Z,X]^\alpha$ and $[Z,X]$ is just another element of the Lorenz group plus scaling if
  $Z$ and $X$ are.
\end{proof}

We start with a first estimate that holds everywhere and is obtained just from the $L^\infty$ estimates of \cite{LS06} that follows directly from the $L^2$ estimates there:
\begin{Proposition}[First Asymptotics for $\phi$]\label{prop:first asym for phi}
With notation and assumptions as in Lemma \ref{lem:smallness of energy} with $k\g 7$ the following hold for
$s+\gamma={s_0^\prime}$, and $1/2<s<1$ and $\gamma>0$. The limit
        $$ \lim_{\xi \rightarrow \infty} \big( r e^{i \frac{1}{4\pi}{\mathbf{q}}   \ln (1+r)}\phi\big)\big( \xi, (\xi + q) \omega\big) = \Phi_0(q, \omega)$$
    exists and we have
       \begin{equation*}\label{eq:Lbarphiasymptotic_first}
       \big|\ull \big( r e^{i \frac{1}{4\pi} {\mathbf{q}}   \ln(1+ r)} \phi\big)(t,x) + 2 \p_q \Phi_0(r-t, \omega) \big| \lesa \epsilon\lr{t+r}^{1/2-s}\lr{t-r}^{-1}\lr{(r-t)_+}^{-\gamma}.
       \end{equation*}
      and
      \begin{equation*}\label{eq:phiasymptotic_first}
      \big|\big( r e^{i \frac{1}{4\pi} {\mathbf{q}}   \ln(1+ r)} \phi\big)(t,x) - \Phi_0(r-t, \omega) \big| \lesa \epsilon\lr{t+r}^{1/2-s}\lr{(r-t)_+}^{-\gamma} \end{equation*}
\end{Proposition}

\begin{proof}
We first have the estimate
\begin{equation}
\label{DL_simppeeling}
\sum_{|I| \leq 1} | D_L(rD_Z^I\phi) | \lesssim \epsilon\lr{t+r}^{-1/2-s}\lr{(r-t)_+}^{-\gamma}.
\end{equation}
This follows almost directly from the estimates \eqref{phi_peelinghigher}.

We now take the identity
\begin{equation}
\label{L_identity}
L(re^{i\frac{1}{4\pi}\mathbf{q}\ln(1+r)}D_Z^I\phi) = e^{i\frac{1}{4\pi}\mathbf{q}\ln(1+r)}D_L(rD_Z^I\phi) + ir\left(\frac{\frac{1}{4\pi}\mathbf{q}}{1+r}-A_L\right)e^{i \frac{1}{4\pi} \mathbf{q}\ln(1+r)}D_Z^I\phi.
\end{equation}
It follows that
\begin{equation*}
\sum_{|I| \leq 1}\left|L(re^{i\frac{1}{4\pi}\mathbf{q}\ln(1+r)}D_Z^I\phi)\right| \lesssim \epsilon\lr{t+r}^{-1/2-s}\lr{(r-t)_+}^{-\gamma}.
\end{equation*}
The estimate for the first quantity on the right hand side of \eqref{L_identity} follows directly from \eqref{DL_simppeeling}. The second quantity is slightly more involved. We can replace  the term $\frac{r\mathbf{q}}{1+r}$ with $\mathbf{q}$ without issue using the (very rough) estimate
\[
\frac{1}{1+r} \lesssim \left(\frac{\lr{t-r}}{\lr{t+r}}\right)^{s-1/2}
\]
and bounding the resulting terms using \eqref{phi_peelinghigher}. We next use  \eqref{eqn:AL asymp bound} combined with \eqref{phi_peelinghigher} to show
\[
\left|(\frac{1}{4\pi}\mathbf{q} - rA_L)D_Z^I\phi\right|\lesssim \epsilon^2\left( \lr{t+r}^{\frac{1}{2} - (s+\gamma)}
 +\frac{\lr{t-r}}{\lr{t+r}} S^0(t,r)\ind_{\{ t>r \}}\right)\lr{t+r}^{-1}\lr{t-r}^{1/2-s}\lr{(r-t)_+}^{-\gamma}.
\]
This is bounded by $\epsilon\lr{t+r}^{-1/2-s}\lr{(r-t)_+}^{-\gamma}$, which follows from straightforward computation using the inequality $S^0 \leq \frac{1}{\epsilon}\left(\frac{\lr{t+r}}{\lr{t-r}}\right)^\epsilon$

 Integrating along the lines $t-r, \omega = constant$ from either $t=0$ or $r=0$ gives us the asymptotic limit for $re^{i\frac{1}{4\pi}\mathbf{q}\ln(1+r)}D_Z^I\phi$, which we call $\Phi_{0Z}^I$. Integrating backwards from future null infinity gives the bound
\begin{equation*}
\left|\Phi^I_{0Z} - re^{i\frac{1}{4\pi}\mathbf{q}\ln(1+r)}D_Z^I\phi\right| \lesssim \epsilon\lr{t+r}^{1/2-s}\lr{(r-t)_+}^{-\gamma},
\end{equation*}
We expand $\underline{L}$ in our Lorenz fields close to and far from the light cone which gives us our result.
\end{proof}

As in mentioned in introduction, we observe that we can write the equation for $\phi$ as
    $$ \Box \phi =  i A_L \ull \phi + i A_{\ull} L\phi - \frac{i}{r} \omega^k A^j \Omega_{kj} \phi + A^\alpha A_\alpha \phi. $$
    Here the first term in the right only decays like $t^{-2}$ along the light cone where
    the other terms decay  like $t^{-3}$ or $t^{-3}\ln{t}$ along the light cone. We therefore want to remove the first term.
Decomposing $\Box$ with respect to the null frame, we conclude that
     \begin{equation} \label{eqn:box phi in frame} L \ull ( r \phi) + i  {\mathbf{q}} r^{-1}  \ull (r\phi) - r^{-1} \triangle_\omega \phi=  i( {\mathbf{q}}   -  r A_L ) \ull \phi -i  {\mathbf{q}} r^{-1}\phi - i r A_{\ull} L \phi +  i  \omega^k A^j \Omega_{kj} \phi - r A^\alpha A_\alpha \phi.
     \end{equation}
The key point is the decay estimates derived in the previous sections, together with the estimate for the scalar field $\phi$ obtained in \cite{LS06}, see Proposition \ref{prop:exter decay for phi} imply that the right hand side decays at least of the order $t^{-1/2-s-\gamma}\ln{t}$ along the light cone, where $s+\gamma>1/2$ along the light cone. The right hand side as well as $r^{-1} \triangle_\omega \phi$ are therefore integrable
in the direction of the outgoing light cone so multiplying by the integrating factor $e^{i \mathbf{q} \ln{r}}$ and integrating in the $L$ direction gives that $\ull(r\phi)$ is bounded and has a limit.

Before we prove the asymptotics we first prove improved decay estimates in the region $r<t$. To avoid the singularity at the origin we modify the approach slightly and multiply with
$e^{i\frac{1}{4\pi}\mathbf{q}\ln{(1+t)}}$. We have
\begin{Proposition} \label{prop:improvedinteriordecayforfield} With notation and assumptions as in Lemma \ref{lem:smallness of energy} with $k \g 7$ the following hold for
$s+\gamma={s_0^\prime}$, and $1/2<s<1$ and $\gamma>0$.
Let $\theta(r,t)=\ln{(1+t)}\,\ind_{\{r<t+1\}}+\ln{r}\,\ind_{\{r>t+1\}}$. Then
(as distributions)
\begin{equation*}
|\Box \big(e^{i\frac{1}{4\pi}\mathbf{q}\theta(r,t)} \phi\big)|
 \lesssim \big(\epsilon \lr{t+r}^{-\frac{3}{2} - (s+\gamma)}\lr{t-r}^{-\frac{1}{2} - s}
 +\epsilon \lr{t+r}^{-3}\lr{t-r}^{\frac{1}{2} - s}S^0\big)\langle (r-t)_+\rangle^{-\gamma}.
\end{equation*}
Moreover
 \begin{equation}\label{eqn:better decay of phi}
       |\phi\,| \lesa  \epsilon\lr{t+r}^{-1} \lr{t-r}^{\frac{1}{2}-s-\gamma}
       .\end{equation}
\end{Proposition}
\begin{proof}
We first observe that
\begin{equation*}
\Box \big(e^{i\frac{1}{4\pi} \mathbf{q}\ln{(1+t)}} \phi\big)=e^{i\frac{1}{4\pi}\mathbf{q}\ln{(1+t)}}\Big( i\big( A_L\! -\frac{\frac{1}{4\pi}\mathbf{q}}{1\!+t}\big)\ull \phi + i\big( A_{\ull}\!-\frac{\frac{1}{4\pi}\mathbf{q}}{1\!+t}\big) L\phi
-\frac{i}{r} \omega^k \! A^j \Omega_{kj}\phi+\big( A^\alpha\! A_\alpha \! +\frac{i\frac{1}{4\pi}\mathbf{q}\!+(\frac{1}{4\pi}\mathbf{q})^2}{(1\!+t)^2} \big)\phi\Big)
\end{equation*}
Hence when $r<t+1$
\begin{multline*}
\big|\Box \big(e^{i\frac{1}{4\pi}\mathbf{q}\ln{(1+t)}} \phi\big)\big|\lesssim
 \Big( \epsilon \lr{t+r}^{-\frac{1}{2} - (s+\gamma)}
 +\epsilon \frac{\lr{t-r}}{\lr{t+r}^2} S^0(t,r)\Big) |\ull \phi|
 +\epsilon \frac{S^0(t,r)}{\lr{t+r}} \big(|L \phi|+|\slashed{\partial}\phi|\big)
 +\frac{|\phi|}{\lr{t+r}^2}\\
 \lesssim \epsilon \lr{t+r}^{-\frac{3}{2} - (s+\gamma)}\lr{t-r}^{-\frac{1}{2} - s}
 +\epsilon \lr{t+r}^{-3}\lr{t-r}^{\frac{1}{2} - s}S^0.
\end{multline*}
Similarly
\begin{equation*}
\Box \big(e^{i\frac{1}{4\pi}\mathbf{q}\ln{r}} \phi\big)\!=e^{i\frac{1}{4\pi}\mathbf{q}\ln{r}}\Big( i\big( A_L\! -\frac{\mathbf{q}}{4\pi r}\big)\ull \phi + i\big( A_{\ull}\!+\frac{\mathbf{q}}{4\pi r}\big) L\phi
-\frac{i}{r} \omega^k\! A^j \Omega_{kj}\phi+\big( A^{\!\alpha}\! A_\alpha \! +\frac{-i\frac{1}{4\pi}\mathbf{q}\!\!+(\frac{1}{4\pi}\mathbf{q})^2}{r^2} \big)\phi\Big)
\end{equation*}
Hence when $r>t+1$
\begin{multline*}
\big|\Box \big(e^{i\frac{1}{4\pi}\mathbf{q}\ln{r}} \phi\big)\big|\lesssim
  \epsilon \lr{t+r}^{-\frac{1}{2} - (s+\gamma)} |\ull \phi|
 +\epsilon \frac{S^0(t,r)}{\lr{t+r}} \big(|L \phi|+|\slashed{\partial}\phi|\big)
 +\frac{|\phi|}{\lr{t+r}^2}\\
 \lesssim \Big(\epsilon \lr{t+r}^{-\frac{3}{2} - (s+\gamma)}\lr{t-r}^{-\frac{1}{2} - s}
 +\epsilon \lr{t+r}^{-3}\lr{t-r}^{\frac{1}{2} - s}S^0\Big)\langle r-t\rangle^{-\gamma}.
\end{multline*}
The estimate \eqref{eqn:better decay of phi} now follows from Lemma \ref{lem:inhomwaveeqdecay}.
\end{proof}

We now turn to proving the asymptotics:
\begin{Proposition}[Null Asymptotics for $\phi$]\label{prop:exterior asym for phi}
With notation and assumptions as in Lemma \ref{lem:smallness of energy} with $k \g 7$ the following hold for
$s+\gamma={s_0^\prime}$, and $1/2<s<1$ and $\gamma>0$.  Then the limit
        $$ \lim_{t \rightarrow \infty} \big( r e^{ i\frac{1}{4\pi} {\mathbf{q}}   \ln (1+r)}\phi\big)\big( t, (t + q) \omega\big) = \Phi_0(q, \omega)$$
    exists and in the region $2r>t$ we have the bounds
      $$ \big|\ull \big( r e^{i {\frac{1}{4\pi} \mathbf{q}}   \ln(1+ r)} \phi\big)(t,x) + 2 \p_q \Phi_0(r-t, \omega) \big| \lesa \epsilon^2\Big( \lr{t+r}^{\frac{1}{2} - (s+\gamma)} +  \frac{\lr{t-r}}{\lr{t+r}} S^0(t,r) \Big) \lr{t-r}^{ - \frac{1}{2} - s} \lr{(r-t)_+}^{-\gamma} $$
      and
        $$\big|\big( r e^{i {\frac{1}{4\pi} \mathbf{q}}   \ln(1+ r)} \phi\big)(t,x) - \Phi_0(r-t, \omega) \big| \lesa \lr{t+r}^{\frac{1}{2} - (s+\gamma)} +  \lr{t+r}^{-1} \lr{t-r}^{\frac{3}{2}-s} S^0(t,r) \ind_{\{r<t\}}.  $$
\end{Proposition}
\begin{proof} We begin  by observing that from \eref{eqn:box phi in frame} we have the identity
\begin{align*} \label{eq:wavewithphase}
    L \ull \big( r e^{i {\frac{1}{4\pi} \mathbf{q}}   \ln(1+ r)} \phi \big) &= \Big[ L \ull \big( r e^{i {\frac{1}{4\pi} \mathbf{q}}   \ln(1+ r)} \phi \big) - e^{i\frac{1}{4\pi} \mathbf{q} \ln(1+r)} \big(L \ull( r \phi) + i \frac{1}{4\pi} \mathbf{q} \ull \phi\big)\Big] \\
    &\qquad \qquad \qquad + e^{i\frac{1}{4\pi} \mathbf{q} \ln(1+r)} \big(L \ull ( r \phi) + i  {\frac{1}{4\pi} \mathbf{q}}   \ull \phi - r^{-1} \triangle_\omega \phi\big) + e^{i\frac{1}{4\pi} \mathbf{q} \ln(1+r)} r^{-1} \Delta_{\mathbb{S}^2} \phi \notag \\
        &=  e^{ i \frac{1}{4\pi} \mathbf{q} \ln(1+r)} i(\frac{1}{4\pi} \mathbf{q} - r A_{L})\ull \phi + F \notag
\end{align*}
with
    $$ e^{-i {\frac{1}{4\pi} \mathbf{q}}   \ln(1+ r)} F = \frac{  -{i\frac{1}{4\pi} \mathbf{q}}  }{1+r} \Big( \ull \phi + L(r\phi) + \frac{i \frac{1}{4\pi} \mathbf{q} + 1}{(1+r)^2}\phi\Big)  - i r A_{\ull} L \phi +  i  \omega^k A^j \Omega_{kj} \phi - r A^\alpha A_\alpha \phi  + r^{-1} \Delta_{\mathbb{S}^2} \phi. $$
Since $2r>t$, the bounds \eqref{eqn:decay of phi} give
  \begin{equation*}\label{eq:F}   |F(t,r\omega)| \lesa \frac{|\phi|}{r} +|L\phi|S^0
 + \frac{|\Omega \phi|}{r}+ \frac{|\Omega^2 \phi|}{r} + \frac{|\phi|}{r} S^0\lesa \lr{t+r}^{-2}  \lr{t-r}^{\frac{1}{2}-s} \lr{(r-t)_+}^{-\gamma} S^0.
  \end{equation*}
  Moreover
\begin{equation*}
\big|( {\frac{1}{4\pi} \mathbf{q}}   - r A_L)  \ull \phi\big|\lesssim
\Big(\epsilon \lr{t+r}^{\frac{1}{2} - (s+\gamma)}
 +\epsilon \frac{\lr{t-r}}{\lr{t+r}} S^0(t,r)\ind_{\{ t>r \}}\Big)
 \epsilon\lr{t+r}^{-1} \lr{t-r}^{ - \frac{1}{2} - s} \lr{(r-t)_+}^{-\gamma}.
\end{equation*}
Hence for $2r>t$ we have
\begin{multline*}
\big|( {\frac{1}{4\pi} \mathbf{q}}   - r A_L)  \ull \phi\big|+|F(t,r)|\\
\lesssim   \epsilon^2\lr{t+r}^{-\frac{1}{2} - (s+\gamma)} \lr{t-r}^{ - \frac{1}{2} - s} \lr{(r-t)_+}^{-\gamma}+\epsilon^2\lr{t+r}^{-2} \lr{t-r}^{\frac{1}{2}-s} \lr{(r-t)_+}^{-\gamma}  S^0
\end{multline*}
Applying the identity \eref{eqn:pointwise radial wave}, we deduce that for $t_1>t_0>0$ and $q> -t_0$ we have
    \begin{align}
      \big| \big(& r e^{i {\frac{1}{4\pi} \mathbf{q}}   \ln (1+r)}\phi\big)\big( t_1, (t_1 + q) \omega\big) - \big( r e^{i {\frac{1}{4\pi} \mathbf{q}}   \ln (1+r)}\phi\big)\big( t_0, (t_0 + q) \omega\big)\big|
      \label{eqn:proof of phi asymp ext:phi bound 1}\\
      &\lesa \int_{2 t_0 + q}^{2t_1 + q} \!\! |L \big( r e^{i {\frac{1}{4\pi} \mathbf{q}}   \ln (1+r)}\phi\big)\big( 0 , (\xi + q) \omega\big)| d\xi + \int_{2 t_0 + q }^{2t_1 + q}\!\! \int_{-\xi}^{-q} \!\!\lr{\xi}^{-2} \lr{\eta}^{\frac{1}{2}-s} \lr{(-\eta)_+}^{-\gamma} S^0\big( \tfrac{1}{2}(\xi+\eta), \tfrac{1}{2}(\xi-\eta)\big) d\eta d\xi \notag \\
      &\qquad \qquad + \int_{2 t_0 + q }^{2t_1 + q} \int_{-\xi}^{-q} \lr{\xi}^{-\frac{1}{2} - (s+\gamma)}\lr{\eta}^{-\frac{1}{2}-s} \lr{(-\eta)_+}^{-\gamma} d\eta d\xi. \notag
    \end{align}
The decay of the data $(\phi(0), \p_t \phi(0) )$ immediately gives control over the first integral in \eref{eqn:proof of phi asymp ext:phi bound 1}. A more involved computation gives for every $\epsilon_0>0$ the bounds
    $$ \int_{2 t_0 + q }^{2t_1 + q} \int_{-\xi}^{-q} \lr{\xi}^{-2} \lr{\eta}^{\frac{1}{2}-s} \lr{(-\eta)_+}^{-\gamma} S^0\big( \tfrac{1}{2}(\xi+\eta), \tfrac{1}{2}(\xi-\eta)\big) d\eta d\xi \lesa  \lr{t_0}^{\frac{1}{2} - (s+\gamma)}  + \lr{t_0}^{-1} \lr{q}^{\frac{3}{2}-s} S^0(t_0, t_0+q) \ind_{\{q<0\}} $$
and
    $$ \int_{2 t_0 + q }^{2t_1 + q} \int_{-\xi}^{-q} \lr{\xi}^{-\frac{1}{2} - (s+\gamma)}\lr{\eta}^{-\frac{1}{2}-s} \lr{(-\eta)_+}^{-\gamma} d\eta d\xi
            \lesa \lr{t_0}^{\frac{1}{2} - (s+\gamma)} \lr{q}^{\frac{1}{2} - s} \lr{q_+}^{-\gamma} .$$
Therefore, we conclude that
    \begin{multline*}
       \big| \big( r e^{i {\frac{1}{4\pi} \mathbf{q}}   \ln (1+r)}\phi\big)\big( t_1, (t_1 + q) \omega\big) - \big( r e^{i {\frac{1}{4\pi} \mathbf{q}}   \ln (1+r)}\phi\big)\big( t_0, (t_0 + q) \omega\big)\big|\\
       \lesa \lr{t_0}^{\frac{1}{2} - (s+\gamma)} +  \lr{t_0}^{-1} \lr{q}^{\frac{3}{2}-s} S^0(t_0, t_0+q) \ind_{\{q<0\}}
    \end{multline*}
which implies that the limit
    $$\lim_{\xi \to \infty} \big(r e^{i {\frac{1}{4\pi} \mathbf{q}}   \ln (1+r)}\phi\big)\big( \xi, (\xi + q) \omega\big) = \Phi_0(q, \omega) $$
exists and satisfies the claimed bound. To check the limit for $\ull( r e^{i {\frac{1}{4\pi} \mathbf{q}}   \ln (1+r)}\phi)( t_0, (t_0 + q) \omega)$, we note that by integrating along $t+r$, we have for $t_1>t_0>0$ and $q>-t_0$ the identity
     \begin{multline*}  \ull \big( r e^{i {\frac{1}{4\pi} \mathbf{q}}   \ln (1+r)}\phi\big)\big( t_1, (t_1 + q) \omega\big) - \ull \big( r e^{i {\frac{1}{4\pi} \mathbf{q}}   \ln (1+r)}\phi\big)\big( t_0, (t_0 + q) \omega\big)\\
      = \frac{1}{2} \int_{2t_0 + q}^{2 t_1 + q}  L \ull \big( r e^{i {\frac{1}{4\pi} \mathbf{q}}   \ln (1+r)}\phi\big)\big(\tfrac{1}{2}(\xi-q), \tfrac{1}{2}(\xi+q)\big) d\xi.
     \end{multline*}
In particular, again using the decay bounds obtained above, we see that
    \begin{align*}
      \big| \ull \big( r &e^{i {\frac{1}{4\pi} \mathbf{q}}   \ln (1+r)}\phi\big)\big( t_1, (t_1 + q) \omega\big) - \ull \big( r e^{i {\frac{1}{4\pi} \mathbf{q}}   \ln (1+r)}\phi\big)\big( t_0, (t_0 + q) \omega\big)\big|\\
      &\lesa \int_{2t_0 + q}^{2t_1 + q} \epsilon^2\lr{\xi}^{-\frac{1}{2} - (s+\gamma)} \lr{q}^{ - \frac{1}{2} - s} \lr{q_+}^{-\gamma}+\epsilon^2\lr{\xi}^{-2} \lr{q}^{\frac{1}{2}-s} \lr{q_+}^{-\gamma}  S^0\big( \tfrac{1}{2}(\xi-q), \tfrac{1}{2}(\xi+q)\big) d\xi \\
      &\lesa \epsilon^2\lr{t_0}^{\frac{1}{2} - (s+\gamma)} \lr{q}^{ - \frac{1}{2} - s} \lr{q_+}^{-\gamma}+\epsilon^2 \lr{t_0}^{-1}S^0(t_0,t_0+q) \lr{q}^{\frac{1}{2}-s} \lr{q_+}^{-\gamma}
    \end{align*}
Consequently the limit
    $$\lim_{t \to \infty} \ull \big(r e^{i {\frac{1}{4\pi} \mathbf{q}}   \ln (1+r)}\phi\big)\big( t, (t + q) \omega\big)$$
exists, and satisfies the claimed bounds. Finally, the identity
    $$ \ull G \big( t, (t + q)\omega\big) = - 2 \p_q \big[ G\big( t, (t + q)\omega \big) \big] + L G\big( t, (t+q)\omega\big)$$
together with the additional decay of $L( r e^{i {\frac{1}{4\pi} \mathbf{q}}   \ln (1+r)}\phi)( t, (t + q) \omega)$ in $t$, implies that
        $$ \lim_{t \to \infty} \ull \big(r e^{i {\frac{1}{4\pi} \mathbf{q}}   \ln (1+r)}\phi\big)\big( t, (t + q) \omega\big) = -2 \p_q \Phi(q, \omega).$$
\end{proof}

\section{Asymptotics for $A_{\ull}$}\label{sec:asymp for non-tang A}
The first step is use Proposition \ref{prop:exterior asym for phi} to replace $J_{\ull}$ with its asymptotic along light cones. Define
    $$ \mc{J}_{\ull}(q,\omega) = - 2\Im\Big( \Phi_0(q, \omega) \overline{\p_q \Phi_0(q, \omega)} \Big)$$
where $\Phi_0(q, \omega)$ is as in Proposition \ref{prop:exterior asym for phi}. We want to express $J_{\ull}$ in terms of $\mc{J}_{\ull}$ and a remainder which has additional decay. To this end, we first observe that
    \begin{align*}
       r^2 J_{\ull} - \mc{J}_{\ull} &=  \Im\Big(  r e^{i {\frac{1}{4\pi} \mathbf{q}}   \ln(1+r)} \phi \overline{ \ull( r e^{ i {\frac{1}{4\pi} \mathbf{q}}   \ln(1+r)} \phi) }\Big)
                + r^2 \Big( \frac{\frac{1}{4\pi} \mathbf{q}}{1+r} - A_{\ull}\Big) |\phi|^2  - \mc{J}_{\ull}\\
                &= \Im\Big( \big( r e^{i\frac{1}{4\pi} \mathbf{q} \ln(1+r)} \phi - \Phi_0) \overline{ \ull ( r e^{i\frac{1}{4\pi} \mathbf{q} \ln(1+r)} \phi)}\Big) - \Im\Big(r e^{i\frac{1}{4\pi} \mathbf{q} \ln(1+r)} \phi \overline{\big( \ull(r e^{i\frac{1}{4\pi} \mathbf{q} \ln(1+r)} \phi)  + 2 \p_q \Phi_0 \big)}\Big) \\
                &\qquad \qquad - \Im\Big(\big( r e^{i\frac{1}{4\pi} \mathbf{q} \ln(1+r)} \phi - \Phi_0)\overline{\big( \ull(r e^{i\frac{1}{4\pi} \mathbf{q} \ln(1+r)} \phi)  + 2 \p_q \Phi_0 \big)} \Big) + r^2 \Big( \frac{\frac{1}{4\pi} \mathbf{q}}{1+r} + A_{\ull}\Big) |\phi|^2
    \end{align*}
An application of Propositions \ref{prop:weak decay of A}, \ref{prop:exter decay for phi}, and \ref{prop:exterior asym for phi} then gives:
\begin{Lemma}\label{lem:Jest} With notation and assumptions as in Lemma \ref{lem:smallness of energy} with $k \g 7$ the following hold for
$s+\gamma={s_0^\prime}$, and $1/2<s<1$ and $\gamma>0$. We have
    \begin{equation*} \big|\,r^2 J_{\ull}  - \mc{J}_\ull\big|
            \lesa \epsilon^2\lr{t+r}^{{1}/{2} - (s+\gamma)} \lr{t-r}^{-\frac{1}{2} - s} \lr{(r-t)_+}^{-\gamma} + \epsilon^2\lr{t+r}^{-1} \lr{t-r}^{1-2s} \lr{(r-t)_+}^{-2\gamma} S^0.
    \end{equation*}
\end{Lemma}
In particular, again applying Proposition \ref{prop:weak decay of A}, we see that for $2r>t$ we have
         \begin{equation}\label{eqn:decay of correction to rhs for A Lbar}
            \begin{split} &\Big| L \ull ( r A_\ull) - \frac{1}{r} \mc{J}_\ull \Big|\\
                &\lesa  \epsilon\lr{t+r}^{-2}\big( \lr{r-t}^{\frac{1}{2} - (s+\gamma)} + \lr{(r-t)_+}^{1-2(s+\gamma)} S^0\big) +  \epsilon\lr{t+r}^{-\frac{1}{2} - (s+\gamma)} \lr{t-r}^{-\frac{1}{2} - s} \lr{(r-t)_+}^{-\gamma}.
            \end{split}
         \end{equation}
In other words we can write $L\ull( r A_\ull) = \frac{1}{r} \mc{J}_\ull + \text{better}$. Applying the formula \eref{eqn:pointwise radial wave} then gives the following.

\begin{Proposition}\label{prop:exterior asymp for A Lbar}
With notation and assumptions as in Lemma \ref{lem:smallness of energy} with $k \g 7$ the following hold for
$s+\gamma={s_0^\prime}$, and $1/2<s<1$ and $\gamma>0$.
Define
     $$ A^{mod}_{\ull}(t,r\omega) = A_{\ull}\big( t, r\omega\big)
            -\frac{1}{2r} \int_{r-t}^\infty \mc{J}_{\ull}(\eta, \omega) \ln\Big( \frac{\eta + t+r}{\eta + t-r}\Big) d\eta. $$
Then the limit
    $$ \lim_{t \to \infty} ( r A^{mod}_{\ull})\big( t, (t + q) \omega \big) = \mc{A}_{\ull}(q, \omega)$$
exists and satisfies
    $$ \big| ( r A^{mod}_{\ull})(t,r\omega) - \mc{A}_{\ull}(r-t, \omega) \big| \lesa \epsilon \lr{t+r}^{\frac{1}{2}-(s+\gamma)} +  \epsilon\frac{\lr{t-r}}{\lr{t+r}} S^0 \ind_{\{t>r\}}.$$
\end{Proposition}
\begin{proof}
We begin by claiming that, for $t_1>t_0>0$ and $q>-t_0$, we have the identity
    \begin{equation}\label{eqn:prop asym for A Lbar:main ident}
        \begin{split}
          \big| (r A^{mod}_{\ull}&)\big(t_1, (t_1 + q)\omega \big) - ( r A^{mod}_{\ull}) \big( t_0, (t_0 + q)\omega\big) \big| \\
                &= \frac{1}{4} \int_{2 t_0 + q}^{2t_1 + q} \int_{q}^\xi \Big( L \ull (r A_{\ull}) - \frac{1}{r} \mc{J}_{\ull} \Big) \Big( \frac{\xi - \eta}{2}, \frac{\xi+\eta}{2}\Big) d\eta d\xi + \frac{1}{2} \int_{2t_0 + q}^{2t_1 + q} \int_\xi^\infty \frac{1}{\xi+\eta} \mc{J}_{\ull}(\eta, \omega) d\eta d\xi.
        \end{split}
    \end{equation}
We leave this identity for the moment, and turn to the problem of bounding the righthand side of \eref{eqn:prop asym for A Lbar:main ident}. The second integral is straightforward, since by Propositions \ref{prop:exterior asym for phi} and \ref{prop:exter decay for phi} we have  $ |\mc{J}_{\ull}(\eta, \omega)| \lesa  \epsilon\lr{\eta}^{-2s}\lr{(\eta)_+}^{-2\gamma}$ and hence
    $$ \int_{2t_0 + q}^{2t_1 + q} \int_\xi^\infty \frac{1}{\xi+\eta} |\mc{J}_{\ull}(\eta, \omega)| d\eta d\xi \lesa  \epsilon\lr{t_0}^{1-2(s+\gamma)}.$$
On the other, to bound the first integral in \eref{eqn:prop asym for A Lbar:main ident}, we apply \eref{eqn:decay of correction to rhs for A Lbar} which gives
    \begin{align*}
      \frac{1}{4} \int_{2 t_0 + q}^{2t_1 + q} \int_{q}^\xi& \Big|\Big( L \ull (r A_{\ull}) - \frac{1}{r} \mc{J}_{\ull} \Big) \Big( \frac{\xi - \eta}{2}, \frac{\xi+\eta}{2}\Big)\Big| d\eta d\xi \\
            &\lesa  \epsilon\int_{2 t_0 + q}^{2t_1 + q} \int_{q}^\xi \lr{\xi}^{-2}\big( \lr{\eta}^{\frac{1}{2} - (s+\gamma)} + \lr{\eta_+}^{1-2(s+\gamma)} S^0\big) + \lr{\xi}^{-\frac{1}{2} - (s+\gamma)} \lr{\eta}^{-\frac{1}{2} - s} \lr{\eta_+}^{-\gamma} d\eta d\xi \\
            &\lesa   \epsilon\lr{t_0}^{\frac{1}{2} - (s+\gamma)} +  \epsilon \lr{t_0}^{-1} \lr{q} S^0 \ind_{\{q<0\}}.
    \end{align*}
Thus letting $t_0 \to \infty$ in \eref{eqn:prop asym for A Lbar:main ident}, we see that the limit exists, and satisfies the claimed bound. It remains to check the identity \eref{eqn:prop asym for A Lbar:main ident}, but this is a consequence of \eref{eqn:pointwise radial wave} together with
    \begin{align*}
      \frac{1}{2} \int_{q}^\infty \mc{J}_{\ull}(\eta, \omega) &\ln\Big( \frac{\eta + 2t_1 + q}{\eta - q}\Big) d\eta  - \frac{1}{2}
      \int_{q}^\infty \mc{J}_{\ull}(\eta, \omega) \ln\Big( \frac{\eta + 2t_0 + q}{\eta -q}\Big) d\eta \\
      &=  \frac{1}{2} \int_q^\infty \mc{J}_{\ull}(\eta, \omega) \ln\Big( \frac{\eta + 2 t_1 + q}{\eta + 2 t_0 + q} \Big) d\eta \\
      &=  \frac{1}{2} \int_{2t_0+ q}^{2t_1 + q} \int_q^\xi \frac{1}{\xi+\eta} \mc{J}_{\ull}(\eta, \omega)d\eta d\xi +
             \frac{1}{2} \int_{2t_0+ q}^{2t_1 + q} \int_\xi^\infty \frac{1}{\xi+\eta} \mc{J}_{\ull}(\eta, \omega) d\eta d\xi.
    \end{align*}

\end{proof}

\section{Asymptotics in the Interior}\label{sec:interior asymp}
Here we give the proof of Theorem \ref{thm:asymp in interior}. The proof uses a proposition from \cite{L17} on the wave equation with asymptotic sources (namely Proposition 9.3 below), but is otherwise self contained.

Recall the decomposition $A_\mu^1 = A_\mu^2 + A_\mu^0$ from the proof of Proposition \ref{prop:weak decay of A}, where $A^0_\mu$ solves the homogeneous wave equation with the same data as $A^1_{\mu}$, and $A^2_\mu$ solves the inhomogeneous equation with vanishing initial data. By the theory of linear homogeneous equations we have the estimate
\[
|Z^IA^0_\mu| \lesssim \epsilon \lr{t+r}^{-1}\lr{t-r}^{1/2-s-\gamma}.
\]
It follows that in the interior $r < c t$, with $c < 1$, we have the estimate
\[
|tA_\mu^0(t,r\omega)| \lesssim \epsilon\lr{t-r}^{1/2-s-\gamma}.
\]
Consequently, since the correction to $A_\mu$ only plays a role in the exterior region $t<|x|$, we conclude that for $|y|<1$ we have
    $$ \lim_{t\to \infty} tA_\mu(t, ty) = \lim_{t\to \infty} t A^2_\mu(t,ty).$$
In particular, to prove Theorem \ref{thm:asymp in interior}, it is enough to show the following result:
\begin{Theorem}\label{th:int_asymptotics} With notation and assumptions as in Lemma \ref{lem:smallness of energy} with $k \g 7$ the following hold for
$s+\gamma={s_0^\prime}$, and $1/2<s<1$ and $\gamma>0$.
Given $\omega' \in \mathbb{S}^2$, $c < 1$, we have the limit
\[
\lim_{t \to \infty}tA^2_\mu(t, ct \omega') = \frac{1}{4\pi}\int_{-\infty}^\infty\int_{\mathbb{S}^2}\frac{\mc{J}_\mu(q, \omega)}{1-c\lr{\omega', \omega}}\, dS(\omega) \, dq,
\]
where $\mc{J}_\mu$ is the asymptotic source term
\begin{equation}\label{mcJdef}
\mc{J}_\mu(q,\omega) = L_\mu(\omega)\Im\left(\Phi_0(q, \omega)\overline{\partial_q\Phi_0(q, \omega)}\right).
\end{equation}
Additionally, for $t \geq 1$, we have the following bound on the difference:
\begin{equation*}
\left|tA^2_\mu(t, c t \omega') - \frac{1}{4\pi}\int_{-\infty}^\infty\int_{\mathbb{S}^2}\frac{\mc{J}_\mu(q, \omega)}{1-c\lr{\omega', \omega}}\, dS(\omega) \, dq\right| \lesssim \epsilon^2\lr{t-r}^{1/2-s-\gamma} + \epsilon^2\left(1+\ln\left(\frac{t+r}{t-r}\right)\right) |t-r|^{1-2s}
\end{equation*}
The implicit constant in $\lesssim$ in particular does not depend on the value of $c$.
\end{Theorem}
\begin{proof}
We break this down into two steps.

Our first objective is to approximate $A^2_\delta$ by an explicit solution $A_\delta^{ex}$. Our method for this is the construction of an intermediate approximation, $A^{as}_\delta$. This represents a solution to the wave equation whose source is an asymptotic approximation to the current vector $J_\delta$.

We will in particular use the inequality
\begin{equation*}
|A^2_\mu - A^{ex}_\mu| \leq |A^2_\mu - A^{as}_\mu| + |A^{as}_\mu - A^{ex}_\mu|.
\end{equation*}

We note that in our null frame, $\mc{J}_{\ull} = \mc{J}_\mu\ull^\mu$ is consistent with its value in the proof of Proposition \ref{prop:exterior asymp for A Lbar}, and $\mc{J}_L = \mc{J}_{S_B} = 0$. We additionally define the asymptotic source approximation
\begin{equation*}
J^\infty_\mu = r^{-2}\mc{J}_{\mu}\chi_0\left(\frac{\lr{r-t}}{r+t}\right).
\end{equation*}

Here, $\chi_0$ is a smooth decreasing cutoff such that
\[
\chi_0(s) = \begin{cases}
1 & s\leq 1/2, \\
0 & s \geq 3/4.
\end{cases}
\]
 The presence of the cutoff function allows us to characterize the asymptotic behavior of the source term close to the light cone without running into issues at $r=0$. We now consider the following estimate:
\begin{equation}\label{eq:JmuEst}
|J_\mu - J^\infty_\mu| \lesssim \epsilon^2\lr{t+r}^{-3/2-s-\gamma}\lr{t-r}^{-1/2-s}\lr{(r-t)_+}^{-\gamma} + \epsilon^2\lr{t+r}^{-3}\lr{t-r}^{1-2s}\lr{(r-t)_+}^{-2\gamma}S^0.
\end{equation}

This follows almost directly from Lemma \ref{lem:Jest}. In particular, the estimate in the support of $\chi_0$ for $J_\ull$ directly follows from the estimate, the estimate for all other components and in the exterior follows from \eqref{eq:Jestiamte} for the $J_\ull$ outside the support of $\chi_0$, where in particular we have $r+t \lesssim \lr{r-t}$, and from \eqref{eq:goodJestiamte} for other components.

 Therefore, given the equation
\begin{equation*}
\Box (A^2_\mu - A^{as}_\mu) = - J_\mu +J^\infty_\mu,
\end{equation*}
with initial conditions
\[(A^2_\mu - A^{as}_\mu)(x, 0) = \partial_t(A^2_\mu - A^{as}_\mu)(x, 0) =0,
\]
we have the following estimate for $r < c t, c< 1$, which follows directly from  Lemma \ref{lem:inhomwaveeqdecay} using the estimate \eqref{eq:JmuEst} :
\begin{lemma}\label{Aar est} With notation and assumptions as in Lemma \ref{lem:smallness of energy} with $k \g 7$ the following hold for
$s+\gamma={s_0^\prime}$, and $1/2<s<1$ and $\gamma>0$.
\begin{equation*}
|A^2_\mu - A^{as}_\mu| \lesssim \epsilon^2\lr{t+r}^{-1}\lr{t-r}^{\max(1-2s, 1/2-s-\gamma)}.
\end{equation*}
\end{lemma}
This in particular does not contribute to the long-range asymptotic behavior of $A_\mu$.

 We now restate Proposition 23 from  \cite{L17}, which will be of use when solving the wave equation with the asymptotic source term $J^\infty$:

\begin{Proposition}[{\cite[Proposition 23]{L17}}]
\label{prop:asymptotic source estimate}
Take
\begin{equation}
\label{def:Aex}
A_\mu^{ex}(t,x) = \int_{|x|-t}^{\infty}\frac{1}{4\pi}\int \frac{\mc{J}_\mu(q, \omega)}{t+q-\lr{x, \omega}}\, dS(\omega)\, \chi_0\left(\frac{\lr{q}}{t+|x|}\right)\, dq,
\end{equation}
and $A^{as}_\mu$ solving the wave equation
\begin{equation*}
\Box A^{as}_\mu = -J^\infty_\mu
\end{equation*}
with vanishing initial data, where
for some $1/2<s<1$
\begin{equation*}
|\mc{J}_\mu(q, \omega)|\lesssim \langle q\rangle^{-2s}
\end{equation*}
We have the estimate
\begin{equation*}
|A^{as}_\mu - A^{ex}_\mu| \lesssim \epsilon^2\lr{t+r}^{-1}\lr{t-r}^{1-2s}.
\end{equation*}
\end{Proposition}

Combining Lemma \ref{Aar est} and Proposition \ref{prop:asymptotic source estimate} gives us our first estimate:
\begin{Lemma} With notation and assumptions as in Lemma \ref{lem:smallness of energy} with $k \g 7$ the following hold for
$s+\gamma={s_0^\prime}$, and $1/2<s<1$ and $\gamma>0$.
Given $A^2_\mu$ satisfying the equation
\[
\Box A^2_\mu = -J_\mu
\]
with initial data equal to zero
and $A_\mu^{ex}$ as defined in \eqref{def:Aex}, where $\mc{J}$ is the asymptotic limit of $J$ as defined in \eqref{mcJdef}, we have the estimate
\begin{equation*}
|A_\mu - A_\mu^{ex}| \lesssim \epsilon^2\lr{t+r}^{-1}\lr{t-r}^{\max(1-2s, 1/2-s-\gamma)}
\end{equation*}
\end{Lemma}

Now we look at the asymptotic behavior of $A_\mu^{ex}$. We for now focus on lines $r = ct, \omega = constant$, and split into the cases $c \leq 1/8$ and $1/8 < c < 1$. The proofs for these cases are very similar and we in general only distinguish them when a factor of $r^{-1}$ appears. We first integrate in $\omega$. For $r = ct$, we wish to show that this integral approaches the integral centered at 0; i.e. we wish to show that, for $x =ct \omega'$, for fixed constant $c$ and fixed $\omega' \in \mathbb{S}^2$, the limit
\begin{equation*}
\mc{A}_\mu^{int}\left(c, \omega'\right) = \lim_{t \to \infty}\left(\frac{1}{4\pi}t\int_{(c - 1)t}^{\infty}\int_{\mathbb{S}^2}\frac{\mc{J}_{\mu}(q,\omega)}{t(1 + q/t -c\lr{\omega', \omega})} \chi_0\left(\frac{\lr{q}}{t+ct}\right)\, dS(\omega)\,dq\right)
\end{equation*}
exists.

We can in fact show that this limit is equal to
\begin{equation*}
\int_{-\infty}^{\infty}\int_{\mathbb{S}^2}\frac{1}{4\pi}\frac{\mc{J}_{\mu}(q,\omega)}{(1 - c\lr{\omega', \omega})}\,dS(\omega)\,dq.
\end{equation*}
We consider only when $t \geq 1$.

 We first take the following useful identity which holds when $|x| < a$:
\begin{equation}\label{eq:angint}
\int_{\mathbb{S}^2}\frac{dS(\omega)}{a-\lr{x, \omega}}= \frac{2\pi}{|x|}\ln\left(\frac{a+|x|}{a-|x|}\right).
\end{equation}

We first set
\begin{equation*}\label{def:Aex10}
A_\mu^{ex,1}(t,x) = \int_{|x|-t}^{\infty}-\frac{1}{4\pi}\int \frac{\mc{J}_\mu(q, \omega)}{t+q-\lr{x, \omega}}\, \chi_1(r,t,q) dS(\omega)\, \, dq,
\end{equation*}
where $\chi_1(r,t,q)=\chi_0\big(\langle q\rangle/(t+r)\big)$, when $q<0$ and $=1$ when $q\geq 0$.
Then it is immediate that
\begin{equation*}
|A^{ex}_\mu-A^{ex,1}_\mu| \lesssim \epsilon^2  |t|^{-2s},
\end{equation*}
since $\lr{q} > \frac12(r+t) \geq |t|$ in the support of $\chi_1-\chi_0$, with the additional condition that $\lr{q} \geq \frac12\lr{t}$ when $t > 1/2$. Next we define
\begin{equation*}\label{def:Aex20}
A_\mu^{ex,2}(t,x) = \int_{|x|-t}^{\infty}\frac{1}{4\pi}\int \frac{\mc{J}_\mu(q, \omega)}{t-\lr{x, \omega}}\,\chi_1(r,t,q) dS(\omega)\, \, dq.
\end{equation*}
The difference bound for this term is slightly more complicated. We treat it as follows:

\begin{lemma}
\label{lem:asymptotic source estimate}
Take
\begin{equation*}
\label{def:Aex1}
A_\mu^{ex,1}(t,x) = \int_{|x|-t}^{\infty}\frac{1}{4\pi}\int \frac{\mc{J}_\mu(q, \omega)}{t+q-\lr{x, \omega}}\, \chi(r,t,q) dS(\omega)\, \, dq,
\end{equation*}
\begin{equation*}
\label{def:Aex2}
A_\mu^{ex,2}(t,x) = \int_{|x|-t}^{\infty}\frac{1}{4\pi}\int \frac{\mc{J}_\mu(q, \omega)}{t-\lr{x, \omega}}\,\chi(r,t,q) dS(\omega)\, \, dq,
\end{equation*}
and suppose that
\begin{equation*}
|\mc{J}_\mu(q, \omega)|\lesssim \langle q\rangle^{-2s}, \qquad |\chi(r,t,q)|\lesssim 1
\end{equation*}
Then for $r<t$
\begin{equation*}
|A^{ex,1}_\mu-A^{ex,2}_\mu| \lesssim t^{-1} |t-r|^{1-2s}.
\end{equation*}
\end{lemma}
\begin{proof}

We consider the case $\frac{|x|}{t} \geq \frac18$; the far interior case follows from the fact that in that region we can use the following inequality which holds in the support of $\chi$:
\[
\frac{1}{t+q-\lr{x, \omega}} - \frac{1}{t-\lr{x, \omega}}\lesssim \frac{q}{t^2},.
\]
which follows straightforwardly from the fact that both functions are bounded below by $\frac{1}{32t}$ in this region.

In the region $\frac{|x|}{t} \geq \frac18$ we introduce coordinates on the sphere with $\omega_1=\langle x/r,\omega\rangle$ and integrate over the other angles we get $dS(\omega)=d\omega_1$, and hence
\begin{multline*}
|A^{ex,1}_\mu-A^{ex,2}_\mu|\\
\lesssim
\int_{r-t}^0\int_{S^2} \Big( \frac{1}{t+q-\lr{x, \omega}}-\frac{1}{t-\lr{x, \omega}}\Big)\, dS(\omega)  \frac{dq}{\langle q\rangle^{2s}}
+\int_{0}^\infty\int_{S^2} \Big( \frac{1}{t-\lr{x, \omega}}-\frac{1}{t+q-\lr{x, \omega}}\Big)\, dS(\omega) \frac{dq}{\langle q\rangle^{2s}} \\
\lesssim
\int_{r-t}^0\frac{1}{r}
\Big(\ln{\Big(\frac{t+r+q}{t-r+q}\Big)} -\ln{\Big(\frac{t+r}{t-r}\Big)} \Big)\frac{dq}{\langle q\rangle^{2s}}
+\int_{0}^\infty  \frac{1}{r} \Big(\ln{\Big(\frac{t+r}{t-r}\Big)}-\ln{\Big(\frac{t+r+q}{t-r+q}\Big)}\Big)\frac{dq}{\langle q\rangle^{2s}}\\
=\int_{r-t}^0\frac{1}{r}
\Big(\ln{\Big(\frac{t-r}{t-r+q}\Big)} -\ln{\Big(\frac{t+r}{t+r+q}\Big)} \Big)\frac{dq}{\langle q\rangle^{2s}}
+\int_{0}^\infty  \frac{1}{r} \Big(\ln{\Big(\frac{t-r+q}{t-r}\Big)}-\ln{\Big(\frac{t+r+q}{t+r}\Big)}\Big)\frac{dq}{\langle q\rangle^{2s}}\\
\lesssim \int_{r-t}^0\frac{1}{r}
\ln{\Big(\frac{t-r}{t-r+q}\Big)}\frac{dq}{\langle q\rangle^{2s}}
+\int_{0}^\infty  \frac{1}{r} \ln{\Big(\frac{t-r+q}{t-r}\Big)}\frac{dq}{\langle q\rangle^{2s}}\lesssim\frac{1}{r}\frac{1}{|t-r|^{2s-1}},
\end{multline*}
as is seen by changing variable $q=(t-r)q'$.
\end{proof}
With
\begin{equation*}\label{def:Aex30}
A_\mu^{ex,3}(t,x) = \int_{|x|-t}^{\infty}\frac{1}{4\pi}\int \frac{\mc{J}_\mu(q, \omega)}{t-\lr{x, \omega}}\, dS(\omega)\, \, dq,
\end{equation*}
it is also immediate that
\begin{equation*}
|A^{ex,2}_\mu-A^{ex,3}_\mu| \lesssim \epsilon^2  r^{-1}|r-t|^{1-2s}\ln\left(\frac{t+r}{t-r}\right),
\end{equation*}
which follows from the identity \eqref{eq:angint} combined with a change of variables similar to the previous lemma.

Finally defining
\begin{equation*}\label{def:Aexinfinity0}
A_\mu^{ex,\infty}(t,x) = \int_{-\infty}^{\infty}\frac{1}{4\pi}\int \frac{\mc{J}_\mu(q, \omega)}{t-\lr{x, \omega}}\, dS(\omega)\, \, dq,
\end{equation*}

it is also immediate that
\begin{equation*}
|A^{ex,3}_\mu-A^{ex,\infty}_\mu| \lesssim \epsilon^2 r^{-1} |t-r|^{1-2s}\ln\left(\frac{t+r}{t-r}\right),
\end{equation*}
since $q<-(t-r)$ in the difference integral.

Finally, we note that in the far interior, with $r < \frac{t}{2}$ for instance, we have the estimate
\[
r^{-1}\ln\left(\frac{t+r}{t-r}\right) \lesssim t^{-1},
\]
so in particular we do not have a singularity at $r=0$. When $r \geq \frac{t}{2}$ we must use the slightly worse estimate
\[
r^{-1}\ln\left(\frac{t+r}{t-r}\right) \lesssim t^{-1}\ln\left(\frac{t+r}{t-r}\right).
\]

Therefore, we have the estimate
\begin{equation*}
|A_\mu^{ex} - A_\mu^{ex, \infty}| \lesssim \epsilon t^{-1}|t-r|^{1-2s}\left(1+\ln\left(\frac{t+r}{t-r}\right)\right)
\end{equation*}

Our result follows.
\end{proof}

\section{Appendix: The radial estimates}\label{sec:appendix}
Here we briefly state some radial estimates that are slight improvement of estimates in \cite{L17}

\begin{lemma}\label{lem:inhomwaveeqdecay} If $-\Box \phi=F$, with
vanishing data, where
\beqs \label{eq:goodinhomdecay} |F|\leq
\frac{C}{(1+r)(1+t+r)(1+|\,t-r|)^{1+\delta}},\qquad \delta>0
 \eqs
 then with $q_+=r-t$, when $r\geq t$ and $q_+=0$, when $r\leq t$, and $ \langle \, q\,\rangle=\sqrt{1+q^2}$ we have
 \begin{equation}\label{eq:logest1}
 |\phi|\leq\frac{C S^0(t,r)}{(1+t+r)\,(1+q_+)^\delta},\qquad
 \text{where}\quad S^0(t,r)
 =\frac{t}{r}\ln{\Big(\frac{\langle \,t+r\,\rangle}{\langle\,t-r\,\rangle}\Big)}.
 \end{equation}

 On the other hand suppose that for some $\mu>0$
 \begin{equation*}\label{eq:goodinhomdecay2} |F|\leq
\frac{C}{(1+r)(1+t+r)^{1+\mu}(1+|\,t-r|)^{1-\mu}(1+q_+)^{\delta_+}
(1+q_-)^{\delta_-}}.
 \end{equation*}
 Then if $0< \delta_-<\mu,\,\, \, 0\leq
\delta_-\leq \delta_+$ we have
 \beq\label{eq:logest2}
 |\phi|\leq\frac{C}{(1+t+r)(1+q_+)^{\delta_+}(1+q_-)^{\delta_-}}.
 \eq
Finally, if $0<\mu < \delta_-\leq \delta_+$, then
\beq\label{eq:logest3}
|\phi| \leq \frac{C}{(1+t+r)(1+|q|)^\mu(1+q_+)^{\delta_+-\mu}}.
\eq
\end{lemma}
\begin{proof} Let $\overline{F}(t,r)\!=\!\sup_{\omega\in S^2} |F(t,r\omega)|$
and let $F_0\!=\!F H$ where $H\!=\!1$,
when $t\!>\!0$ and $H\!=\!0$, when $t\!<\!0$.
 Since $|F_0|\leq \overline{F}_0$ it follows from
 the positivity of the fundamental solution that $|\phi|\leq
 |\overline{\phi}|$ where $\overline{\phi}$ is the solution of
 $-\Box\overline{\phi}=\overline{F}_0$ with vanishing initial data.
 Since the wave operator is invariant under rotations
 it follows that $\overline{\phi}$ is independent of the angular
 variables so
 $
 (\pa_t-\pa_r)(\pa_t+\pa_r)(r\overline{\phi}(t,r))=r\overline{F}_0.
 $
 If we now introduce new variables $\xi=t+r$ and $\eta=t-r$ and
 integrate over the region $R=\{(\xi,\eta);\,
 -\infty\leq \eta\leq t-r,\, \, t-r\leq \xi \leq t+r\}$
 using that $r\overline{\phi}(t,r)$ vanishes when $\eta=-\infty$ and
 when $r=0$, i.e. $\xi=\eta=t-r$ we obtain
 \beqs
 r\overline{\phi}(t,r)=4\int_{t-r}^{t+r} \int_{-\infty}^{t-r}
 \rho \overline{F}_0(s,\rho) H(s) \, d\eta d\xi,
 \qquad s=\frac{\xi+\eta}{2},\quad \rho=\frac{\xi-\eta}{2}.
 \eqs
 In the first case we have
 \beqs
 r\overline{\phi}(t,r)\leq 4\int_{t-r}^{t+r} \int_{-\xi}^{t-r}
 \frac{H(\xi+\eta) }{(1+|\xi|)^{}(1+|\eta|)^{1+\delta}}
 \, d\eta d\xi.
 \eqs
 If $t\!>\!r$ \eqref{eq:logest1} follows from integrating this,
 since
  $
  \frac{1}{r}\log{\big(\frac{1+t+r}{1+t-r}\big)}
 \!\leq \!\frac{C}{1+t+r} S^0(t,r).
  $
  If $r>t$ then we integrate first in the $\xi$ direction
\beqs
 r\overline{\phi}(t,r)\les \int_{\!-(t+r)}^{-(r-t)}\!\!
 \int_{|\eta|}^{t+r}
 \!\!\!\!\!\frac{ d\xi d\eta }{(1+|\,\xi|)(1+|\eta|)^{1+\delta}}
\les \int_{\!-(t+r)}^{-(r-t)}
 \frac{ \log{\big|\frac{1+t+r}{1+|\eta|}\big|} \, d\eta }
 {(1+|\eta|)^{1+\delta}}
 \les \frac{C}{(1+t+r)^\delta}
 \int_{\frac{1+r-t}{1+r+t}}^1\!\!\!\!\!
 \frac{\ln{\big|\frac{1}{s}\big|}\, ds}
 {\,\,\,s^{1+\delta}\!\!\!},
 \eqs
  and \eqref{eq:logest1} for $r>t$ follows from this.
To prove \eqref{eq:logest2} we must estimate
 \begin{equation*}
 r\overline{\phi}(t,r)\leq 4\int_{t-r}^{t+r} \int_{-\xi}^{t-r}
 \frac{ H(\xi+\eta) \,\, d\eta \,d\xi}{(1+|\xi|)^{1+\mu}
 (1+|\eta|)^{1-\mu}(1+\!\eta_-)^{\delta_+}
 (1+\!\eta_+)^{\delta_-}}.
 \end{equation*}
 If $r>t$ then we integrate first in the $\xi$ direction
\beqs
 r\overline{\phi}(t,r)\les \int_{\!-(t+r)}^{-(r-t)}\!\!\int_{|\eta|}^{t+r}
 \!\!\!\!\!\frac{\ d\xi d\eta }{(1+|\xi|)^{1+\mu}
 (1+\!|\,\eta|)^{1+\delta_+ -\mu}}
\les \int_{\!-(t+r)}^{-(r-t)}
 \frac{\, d\eta }{(1+\!|\,\eta|)^{1+\delta_+}}
\eqs
which is $\les (1+\!|\,t-r|)^{-\delta_+}$ so
\eqref{eq:logest2} for $r>t$ follows.
 If $t>r$ and we set ${\delta}=\min(\delta_-,\,\delta_+)$ it follows from integrating in the $\eta$ direction
  that if
 $1+\delta -\mu<1$ we have
 \beqs
 r\overline{\phi}(t,r)\leq\! 4\!\int_{t-r}^{t+r}
\frac{(1+\!|\xi|)^{\mu-\delta}+(1+\!|\,t-r|)^{\mu-\delta} \!\!\!}{(1+|\xi|)^{1+\mu}}
 \, d\xi
  \leq \frac{C r}{(1+t+r)(1+|t-r|)^{{\delta}}}.
 \eqs
The proof for \eqref{eq:logest3} follows the same approach in the interior, where we instead use the fact that if $1+\delta-\mu > 1$ we have
\[
r\overline{\phi}(t,r)\leq C\int_{t-r}^{t+r}\frac{1}{(1+|\xi|)^{1+\mu}}d\xi
\]
\end{proof}

\begin{lemma}\label{lem:homwaveeqdecay} If $w$ is the solution of
 \begin{equation*}
 -\Box w=0,\quad \quad w\big|_{t=0}=w_0,\quad \partial_t w\big|_{t=0}=w_1
 \end{equation*}
 then for any $0<\gamma<1$;
 \beq\label{eq:homoest}
 (1+t+r)(1+|\,r-t|)^\gamma |w(t,x)|\\
 \les {\sup}_x \big( (1+|x|)^{2+\gamma}
 ( |w_1(x)|+|\,\partial w_0(x)|) +(1+ |x|)^{1+\gamma}|w_0(x)|\big).
 \eq
\end{lemma}
\begin{proof} The proof is an immediate consequence of
Kirchoff's formula
 $$
w(t,x)=t\!\int {\big(w_1(x+t{\omega})+\langle
w^{\prime}_0(x+t{\omega}),{\omega}\rangle \big)\,\frac{dS({\omega})}{4\pi}} +
 \int {w_0(x+t{\omega})\,\frac{dS({\omega})}{4\pi}},
 $$
 where $dS(\omega)$ is the measure on $\bold{S}^2$.
 If $x\!=\!r\bold{e}_1$, where $\bold{e}_1\!=\!(1,0,0)$ then
for $k\!=\!1,2$
 \beqs
 \int\!\! \frac{dS(\omega)/4\pi}{1\!+|r\bold{e}_1\!+t\omega|^{k+\gamma}}
 =\int_{\!-1}^1\frac{d\omega_1/2}{1\!+\big(
 (r\!-t\omega_1)^2\!+t^2(1\!-\omega_1^2)\big)^{(k+\gamma)/2}}
 =\int_0^2\!\! \frac{d s/2}{1\!+\big(
 (r\!-t)^2\!+2rt s\big)^{(k+\gamma)/2}}.
 \eqs
 \eqref{eq:homoest} follows directly if $|r\!-t|\! \geq \! t/2$.
If $t/2\!<\!r\!<\!2t$ we change variables $\tau\!=\!rt s$.
  If $k\!=\!2$ it can be bounded by
  $(rt)^{-1}(1+|r\!-t|)^{-\gamma}$ and if $k\!=\!1$ by
  $(rt)^{-1}(1\!+rt)^{(1-\gamma)/2}$.
 \end{proof}


\begin{thebibliography}{99}


\bibitem{BMS17} L. Bieri, S. Miao and S. Shahshahani {\it Asymptotic properties of solutions of the {M}axwell {K}lein
              {G}ordon equation with small data}, Comm. Anal. Geom. {\bf{25(1)}} (2017) 25--96.

\bibitem{CK90} D. Christodoulou and S. Klainerman {\it Asymptotic properties of linear field equations in {M}inkowski space}, Comm. Pure Appl. Math. {\bf{43(3)}} (1990) 137--199.

\bibitem{EM82a}  Eardley, Douglas M.; Moncrief, \emph{Vincent The
  global existence of Yang-Mills-Higgs fields in
  $4$-dimensional Minkowski space. I. Local existence and smoothness
  properties.}  Comm. Math. Phys.  \textbf{83}  (1982), no. 2, 171--191.


\bibitem{EM82b}  Eardley, Douglas M.; Moncrief, Vincent \emph{The
  global  existence of Yang-Mills-Higgs fields in $4$-dimensional
  Minkowski space. II. Completion of proof.}
Comm. Math. Phys.  \textbf{83}  (1982), no. 2, 193--212.

\bibitem{H97} L. H{\"o}rmander {\it Lectures on nonlinear hyperbolic differential equations} Springer-Verlag (1997).



\bibitem{K18} C. Kauffman,
{\it Global Stability for Charged Scalar Fields in an Asymptotically Flat Metric in Harmonic Gauge}
    Preprint (2018)

\bibitem{KM94}  Klainerman, S.; Machedon, M. \emph{On the
  Maxwell-Klein-Gordon equation with finite energy.}
Duke Math. J.  \textbf{74}  (1994),  no. 1, 19--44.

\bibitem{KWY18} S. Klainerman, Q. Wang and S. Yang {\it Global solution for massive Maxwell-Klein-Gordon equations} preprint (2018).

\bibitem{L90} H. Lindblad, {\it Blow up for solutions of  $\Box \, u \, = \,
|\, u \, |^p$ with small initial data},  Comm. Part. Diff Eq. {\bf 15(6)}
(1990), 757-821.


\bibitem{LR05} H. Lindblad and I. Rodnianski, {\it Global existence for the Einstein
    vaccum equations in wave coordinates.}
    Comm. Math. Phys. 256 (2005), no. 1, 43--110

\bibitem{LS06} H. Lindblad and J. Sterbenz,
{\it Global Stability for Charged Scalar Fields on Minkowski Space}
    IMRP Int. Math. Res. Pap. (2006)

\bibitem{LR10} H. Lindblad and I. Rodnianski, {\it The global stability of
the Minkowski space-time in harmonic gauge.}
 Annals of Math {\bf 171} (2010), no. 3, 1401-1477.

\bibitem{L17} H. Lindblad
{\it On the asymptotic behavior of solutions to Einstein's vacuum equations in wave coordinates.} Comm. Math. Phys. {\bf 353}, (2017), No 1, 135-184

\bibitem{LS17} H. Lindblad and V. Schlue
{\it Scattering from infinity for semilinear models of Einstein's equations satisfying the weak null condition.} Preprint (2017)


\bibitem{P99} M. Psarelli \emph{Asymptotic behavior of the
  solutions of  Maxwell-Klein-Gordon field equations in
  $4$-dimensional Minkowski space.}
  Comm. Partial Differential Equations  \textbf{24}  (1999),  no. 1-2, 223--272.

\bibitem{S91} W-T. Shu {\it Asymptotic properties of the solutions of linear and nonlinear
              spin field equations in {M}inkowski space} Comm. Math. Phys. {\bf{140(3)}} (1991) 449--480.

\bibitem{Y15} S. Yang {\it Decay of solutions of Maxwell-Klein-Gordon equations with large Maxwell field}
preprint (2015)
\end{thebibliography}
\end{document}